\documentclass{amsart}
\usepackage{verbatim, rotating, stmaryrd, bbm, pict2e, ytableau, mathrsfs, amssymb}
\usepackage[colorlinks]{hyperref}
\usepackage[all]{xy}
\usepackage{cleveref}

\usepackage{tikz}
\usetikzlibrary{matrix,arrows,backgrounds}

\makeatletter
\newcommand\@dotsep{4.5}
\def\@tocline#1#2#3#4#5#6#7{\relax
  \ifnum #1>\c@tocdepth
  \else
    \par \addpenalty\@secpenalty\addvspace{#2}%
    \begingroup \hyphenpenalty\@M
    \@ifempty{#4}{%
      \@tempdima\csname r@tocindent\number#1\endcsname\relax
    }{%
      \@tempdima#4\relax
    }%
    \parindent\z@ \leftskip#3\relax \advance\leftskip\@tempdima\relax
    \rightskip\@pnumwidth plus1em \parfillskip-\@pnumwidth
    #5\leavevmode\hskip-\@tempdima #6\relax
    \leaders\hbox{$\m@th
      \mkern \@dotsep mu\hbox{.}\mkern \@dotsep mu$}\hfill
    \hbox to\@pnumwidth{\@tocpagenum{#7}}\par
    \nobreak
    \endgroup
  \fi}
\def\l@subsection{\@tocline{2}{0pt}{20pt}{5pc}{}}
\def\l@subsubsection{\@tocline{2}{0pt}{30pt}{5pc}{}}
\makeatother

\ytableausetup{boxsize=11.5pt}
\ytableausetup{aligntableaux=center}
\definecolor{lred}{RGB}{220,60,60}

\theoremstyle{plain}
\newtheorem{Thm}{Theorem}[section]
\newtheorem*{Thm*}{Theorem}
\newtheorem{Cor}[Thm]{Corollary}

\newtheorem{Lem}[Thm]{Lemma}
\newtheorem{Prop}[Thm]{Proposition}

\theoremstyle{definition}
\newtheorem{Def}[Thm]{Definition}
\newtheorem{Ex}[Thm]{Example}

\theoremstyle{remark}

\DeclareMathOperator{\im}{im}

\DeclareMathOperator{\proj}{Proj}
\DeclareMathOperator{\coker}{coker}
\DeclareMathOperator{\SL}{SL}

\DeclareMathOperator{\rad}{rad}

\DeclareMathOperator{\jtype}{JType}
\DeclareMathOperator{\rank}{rank}
\DeclareMathOperator{\spn}{span}
\DeclareMathOperator{\ad}{ad}
\DeclareMathOperator{\End}{End}
\renewcommand{\hom}{\mathrm{Hom}}

\newcommand{\slt}[1][2]{\mathfrak{sl}_{#1}}

\newcommand{\set}[1]{\left\{#1\right\}}

\newcommand{\gim}[2][\mbox{}]{\im\Theta^{#1}_{#2}}
\newcommand{\gker}[2][\mbox{}]{\ker\Theta^{#1}_{#2}}
\newcommand{\gcoker}[2][\mbox{}]{\coker\Theta^{#1}_{#2}}
\newcommand{\Ho}[1][1]{\mathcal H^{[#1]}}
\newcommand{\F}{\mathscr F}
\newcommand{\K}[1][1]{\ker(\Ho)}
\newcommand{\PG}[1][{\mathfrak g}]{\mathbb P(#1)}
\newcommand{\OPG}[1][\mathfrak g]{\mathcal O_{\mathbb P(#1)}}
\newcommand{\OP}{\mathcal O_{\mathbb P^1}}

\begin{document}

\title[Computations of sheaves associated to $\slt$]{Computations of sheaves associated to the representation theory of $\slt$}

\author{Jim Stark}

\thanks{The author was partially supported by the NSF grant DMS-0953011}

\address{Department of Mathematics\\
   University of Washington\\
   Seattle, WA 98105}

\email{jstarx@gmail.com}

\date{5 September, 2013}

\begin{abstract}
We explicitly compute examples of sheaves over the projectivization of the spectrum of the cohomology of $\slt$.  In particular, we compute $\gker{M}$ for every indecomposable $M$ and we compute $\F_i(M)$ when $M$ is an indecomposable Weyl module and $i \neq p$.  We also give a brief review of the classification of $\slt$-modules and of the general theory of such sheaves in the case of a restricted Lie algebra.
\end{abstract}

\maketitle

\tableofcontents

\setcounter{section}{-1}
\section{Introduction}

Let $\mathfrak g$ be a restricted Lie algebra over an algebraically closed field $k$ of positive characteristic $p$.  Suslin, Friedlander, and Bendel \cite{bfs1paramCoh} have shown that the maximal spectrum of the cohomology of $\mathfrak g$ is isomorphic to the variety of $p$-nilpotent elements in $\mathfrak g$, i.e., the so called restricted nullcone $\mathcal N_p(\mathfrak g)$.  This variety has become an important invariant in representation theory; for example, it can be used to give a simple definition of the local Jordan type of a $\mathfrak g$-module $M$ and consequently of the class of modules of constant Jordan type, a class first studied by Carlson, Friedlander, and Pevtsova \cite{cfpConstJType} in the case of a finite group scheme.  Friedlander and Pevtsova \cite{friedpevConstructions} have initiated what is, in the case of a Lie algebra $\mathfrak g$, the study of certain sheaves over the projectivization, $\PG$, of $\mathcal N_p(\mathfrak g)$.  These sheaves are constructed from $\mathfrak g$-modules $M$ so that representation theoretic information, such as whether $M$ is projective, is encoded in their geometric properties.  Explicit computations of these sheaves can be challenging due not only to their geometric nature but also to the inherent difficulty in describing representations of a general Lie algebra.

The purpose of this paper is to explicitly compute examples of these sheaves for the case $\mathfrak g = \slt$.  The Lie algebra $\slt$ has tame representation type and the indecomposable modules were described explicitly by Alexander Premet \cite{premetSl2} in 1991.  This means there are infinitely many isomorphism classes of such modules, so the category is rich enough to be interesting, but these occur in finitely many parameterized families which allow for direct computations.  We also note that the variety $\PG[\slt]$ over which we wish to compute these sheaves is isomorphic to $\mathbb P^1$.  By a theorem of Grothendieck we therefore know that locally free sheaves admit a strikingly simple description: They are all sums of twists of the structure sheaf.  This makes $\slt$ uniquely suited for such computations.

We begin in \Cref{secRev} with the case of a general restricted Lie algebra $\mathfrak g$.  We will review the definition of $\mathcal N_p(\mathfrak g)$ and its projectivization $\PG$.  We use this to define the local Jordan type of a module $M$.  We define the global operator $\Theta_M$ associated to a $\mathfrak g$-module $M$ and use it to construct the sheaves we are interested in computing.  We will review theorems which not only indicate the usefulness of these sheaves but are also needed for their computation.

In \Cref{secSl2} we begin the discussion of the category of $\slt$-modules.  Our computations are fundamentally based on having, for each indecomposable $\slt$-module, an explicit basis and formulas for the $\slt$ action.  To this end we review Premet's description.  There are four families and for each family we specify not only the explicit basis and $\slt$-action, but also a graphical representation of the module and the local Jordan type of the module.  For the Weyl modules $V(\lambda)$, dual Weyl modules $V(\lambda)^\ast$, and projective modules $Q(\lambda)$ this information was previously known (see for example Benkart and Osborn \cite{benkartSl2reps}) but for the so called non-constant modules $\Phi_\xi(\lambda)$ we do not know if such an explicit description has previously been given.  Thus we give a proof that this description follows from Premet's definition of the modules $\Phi_\xi(\lambda)$.  We also compute the Heller shifts $\Omega V(\lambda)$ of the Weyl modules for use in \Cref{secLieEx}.

In \Cref{secMatThms} we digress from discussing Lie algebras and compute the kernels of four particular matrices.  These matrices, with entries in the polynomial ring $k[s, t]$, will represent sheaf homomorphisms over $\mathbb P^1 = \proj k[s, t]$ but in this section we do not work geometrically and instead consider these matrices to be maps of free $k[s, t]$-modules.  This section contains the bulk of the computational effort of this paper.

In \Cref{secLieEx} we are finally ready to carry out the computations promised.  Friedlander and Pevtsova have computed $\gker{V(\lambda)}$ for the case $0 \leq \lambda \leq 2p - 2$ \cite{friedpevConstructions}.  We compute the sheaves $\gker{M}$ for every indecomposable $\slt$-module $M$.  This computation is essentially the bulk of the work in the previous section; the four matrices in that section describe the global operators of the four families of $\slt$-modules.  We also compute $\F_i(V(\lambda))$ for $i \neq p$ and $V(\lambda)$ indecomposable using an inductive argument.  The base case is that of a simple Weyl module $(\lambda < p)$ and is done by noting that $\F_i(V(\lambda))$ is zero when $i \neq \lambda + 1$ and that $\F_{\lambda + 1}(V(\lambda))$ can be identified with the kernel sheaf $\gker{V(\lambda)}$.  For the inductive step we use the Heller shift computation from \Cref{secSl2} together with a theorem of Benson and Pevtsova \cite{benPevtVectorBundles}.

\section{Jordan type and global operators for Lie algebras} \label{secRev}

In this section we review the definition of the restricted nullcone of a Lie algebra $\mathfrak g$ and of the local Jordan type of a $\mathfrak g$-module $M$.  We also define the global operator associated to a $\mathfrak g$-module $M$ and the sheaves associated to such an operator.  Global operators and local Jordan type can be defined for any infinitesimal group scheme of finite height.  Here we give the definitions only for a restricted Lie algebra $\mathfrak g$ and take $\slt$ as our only example.  For details on the general case as well as additional examples we refer the reader to Friedlander and Pevtsova \cite{friedpevConstructions} or Stark \cite{starkHo1}.

Let $\mathfrak g$ be a restricted Lie algebra over an algebraically closed field $k$ of positive characteristic $p$.  Recall that this means $\mathfrak g$ is a Lie algebra equipped with an additional \emph{$p$-operation} $(-)^{[p]}\colon\mathfrak{g \to g}$ satisfying certain axioms (see Strade and Farnsteiner \cite{stradeFarnModularLie} for details).  Here we merely note that for the classical subalgebras of $\mathfrak{gl}_n$ the $p$-operation is given by raising a matrix to the $p^\text{th}$ power.

\begin{Def}
The restricted nullcone of $\mathfrak g$ is the set
\[\mathcal N_p(\mathfrak g) = \set{x \ \middle| \ x^{[p]} = 0}\]
of $p$-nilpotent elements.  This is a conical irreducible subvariety of the affine space $\mathfrak g$.  We denote by $\PG$ the projective variety whose points are lines through the origin in $\mathcal N_p(\mathfrak g)$.
\end{Def}

\begin{Ex} \label{exNslt}
Let $\mathfrak g = \slt$ and take the usual basis
\[e = \begin{bmatrix} 0 & 1 \\ 0 & 0 \end{bmatrix}, \qquad f = \begin{bmatrix} 0 & 0 \\ 1 & 0 \end{bmatrix}, \quad \text{and} \quad h = \begin{bmatrix} 1 & 0 \\ 0 & -1 \end{bmatrix}.\]
Let $\set{x, y, z}$ be the dual basis so that $\slt$, as an affine space, can be identified with $\mathbb A^3$ and has coordinate ring $k[x, y, z]$.  A $2 \times 2$ matrix over a field is nilpotent if and only if its square
\[\begin{bmatrix} z & x \\ y & -z \end{bmatrix}^2 = (xy + z^2)\begin{bmatrix} 1 & 0 \\ 0 & 1 \end{bmatrix}\]
is zero therefore, independent of $p$, we get that $\mathcal N_p(\slt)$ is the zero locus of $xy + z^2$.

By definition $\PG[\slt]$ is the projective variety defined by the homogeneous polynomial $xy + z^2$.  Let $\mathbb P^1$ have coordinate ring $k[s, t]$ and define a map $\iota\colon\mathbb P^1 \to \PG[\slt]$ via $[s, t] \mapsto [s^2 : -t^2 : st]$.  One then checks that the maps $[1 : y : z] \mapsto [1 : z]$ and $[x : 1 : z] \mapsto [-z : 1]$ defined on the open sets $x \neq 0$ and $y \neq 0$, respectively, glue to give an inverse to $\iota$.  Thus $\PG[\slt] \simeq \mathbb P^1$.
\end{Ex}

To define the local Jordan type of a $\mathfrak g$-module $M$, recall that a combinatorial partition $\lambda = (\lambda_1, \lambda_2, \ldots, \lambda_n)$ is a weakly decreasing sequence of finitely many positive integers.  We say that $\lambda$ is a partition of the integer $\sum_i\lambda_i$, for example $\lambda = (4, 4, 2, 1)$ is a partition of $11$.  We call a partition $p$-restricted if no integer in the sequence is greater than $p$ and we let $\mathscr P_p$ denote the set of all $p$-restricted partitions.  We will often write partitions using either Young diagrams or exponential notation.  A Young diagram is a left justified two dimensional array of boxes whose row lengths are weakly decreasing from top to bottom.  These correspond to the partitions obtained by reading off said row lengths.  In exponential notation the unique integers in the partition are written surrounded by brackets with exponents outside the brackets denoting repetition.

\begin{Ex}
The partition $(4, 4, 2, 1)$ can be represented as the Young diagram
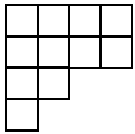
\[\ydiagram{4,4,2,1}\]
or written in exponential notation as $[4]^2[2][1]$.
\end{Ex}

If $A \in \mathbb M_n(k)$ is a $p$-nilpotent ($A^p = 0$) $n \times n$ matrix then the Jordan normal form of $A$ is a block diagonal matrix such that each block is of the form
\[\begin{bmatrix} 0 & 1 \\ & 0 & 1 \\ && \ddots & \ddots \\ &&& 0 & 1 \\ &&&& 0 \end{bmatrix} \qquad (\text{an} \ i \times i \ \text{matrix})\]
for some $i \leq p$.  Listing these block sizes in weakly decreasing order yields a $p$-restricted partition of $n$ called the \emph{Jordan type}, $\jtype(A)$, of the matrix $A$.  Note that conjugation does not change the Jordan type of a matrix so if $T\colon V \to V$ is a $p$-nilpotent operator on a vector space $V$ then we define $\jtype(T) = \jtype(A)$, where $A$ is the matrix of $T$ with respect to some basis.  Finally, note that scaling a nilpotent operator does not change the eigenvalues or generalized eigenspaces so it is easy to see that $\jtype(cT) = \jtype(T)$ for any non-zero scalar $c \in k$.

\begin{Def}
Let $M$ be a $\mathfrak g$-module and $v \in \PG$.  Set $\jtype(v, M) = \jtype(x)$ where $x \in \mathcal N_p(\mathfrak g)$ is any non-zero point on the line $v$ and its Jordan type is that of a $p$-nilpotent operator on the vector space $M$.  The \emph{local Jordan type} of $M$ is the function
\[\jtype(-, M)\colon\PG \to \mathscr P_p\]
so defined.
\end{Def}

When computing the local Jordan type of a module the following lemma is useful.  Recall that the conjugate of a partition is the partition obtained by transposing the Young diagram.

\begin{Lem} \label{lemConj}
Let $A \in \mathbb M_n(k)$ be $p$-nilpotent.  The conjugate of the partition
\[(n - \rank A, \rank A - \rank A^2, \ldots, \rank A^{p-2} - \rank A^{p-1}, \rank A^{p-1}).\]
is $\jtype(A)$.
\end{Lem}

\begin{Ex} \label{exPart}
The conjugate of $[4]^2[2][1]$ is $[4][3][2]^2$.
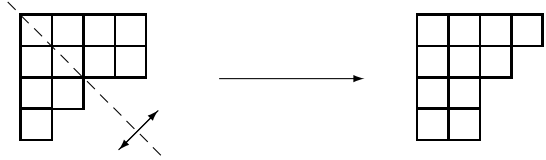
\begin{center}
\begin{picture}(100, 70)(45, 0)
\put(0, 35){\ydiagram{4,4,2,1}}
\put(150, 35){\ydiagram{4,3,2,2}}
\put(75, 37){\vector(1, 0){55}}
\put(-5, 66){\line(1, -1){3}}
\put(0, 61){\line(1, -1){3}}
\put(5, 56){\line(1, -1){3}}
\put(10, 51){\line(1, -1){3}}
\put(15, 46){\line(1, -1){3}}
\put(20, 41){\line(1, -1){3}}
\put(25, 36){\line(1, -1){3}}
\put(30, 31){\line(1, -1){3}}
\put(35, 26){\line(1, -1){3}}
\put(40, 21){\line(1, -1){3}}
\put(45, 16){\line(1, -1){3}}
\put(50, 11){\line(1, -1){3}}
\put(37, 10){\vector(1, 1){15}}
\put(52, 25){\vector(-1, -1){15}}
\end{picture}
\end{center}
\end{Ex}

\begin{Ex} \label{exV2JT}
Assume $p > 2$ and consider the Weyl module $V(2)$, for $\slt$.  This is a $3$-dimensional module where $e$, $f$, and $h$ act via
\[\begin{bmatrix} 0 & 2 & 0 \\ 0 & 0 & 1 \\ 0 & 0 & 0 \end{bmatrix}, \qquad \begin{bmatrix} 0 & 0 & 0 \\ 1 & 0 & 0 \\ 0 & 2 & 0 \end{bmatrix}, \quad \text{and} \quad \begin{bmatrix} 2 & 0 & 0 \\ 0 & 0 & 0 \\ 0 & 0 & -2 \end{bmatrix}\]
respectively.  The matrix
\[A = \begin{bmatrix} 2z & 2x & 0 \\ y & 0 & x \\ 0 & 2y & -2z \end{bmatrix}\]
describes the action of $xe + yf + zh \in \mathfrak g$ on $V(2)$.  For the purposes of computing the local Jordan type we consider $[x : y : z]$ to be an element in the projective space $\PG$.  If $y = 0$ then $xy + z^2 = 0$ implies $z = 0$ and we can scale to $x = 1$.  This immediately gives Jordan type $[3]$.  If $y \neq 0$ then we can scale to $y = 1$ and therefore $x = -z^2$.  This gives
\[A = \begin{bmatrix} 2z & -2z^2 & 0 \\ 1 & 0 & -z^2 \\ 0 & 2 & -2z \end{bmatrix} \quad \text{and} \quad A^2 = \begin{bmatrix} 2z^2 & -4z^3 & 2z^4 \\ 2z & -4z^2 & 2z^3 \\ 2 & -4z & 2z^2 \end{bmatrix}\]
therefore $\rank A = 2$ and $\rank A^2 = 1$.  Using \cref{lemConj} we conclude that the Jordan type is the conjugate of $(3 - 2, 2 - 1, 1) = [1]^3$, which is $[3]$.  Thus the local Jordan type of $V(2)$ is the constant function $v \mapsto [3]$.
\end{Ex}

\begin{Def}
A $\mathfrak g$-module $M$ has \emph{constant Jordan type} if its local Jordan type is a constant function.
\end{Def}

Modules of constant Jordan type will be significant for us for two reasons.  The first is because of the following useful projectivity criterion.

\begin{Thm}[{\cite[7.6]{bfsSupportVarieties}}] \label{thmProjCJT}
A $\mathfrak g$-module $M$ is projective if and only if its local Jordan type is a constant function of the form $v \mapsto [p]^n$.
\end{Thm}

For the second note that when $\mathfrak g$ is the Lie algebra of an algebraic group $G$, the adjoint action of $G$ on $\mathfrak g$ induces an action of $G$ on the restricted nullcone and hence on $\PG$.  One can show that the local Jordan type of a $G$-module (when considered as a $\mathfrak g$-module) is constant on the $G$-orbits of this action.  The adjoint action of $\SL_2$ on $\PG[\slt]$ is transitive so we get the following.

\begin{Thm}[{\cite[2.5]{cfpConstJType}}] \label{thmRatCJT}
Every rational $\slt$-module has constant Jordan type.
\end{Thm}

Next we define the global operator associated to a $\mathfrak g$-module $M$ and the sheaves associated to such an operator.  Let $\set{g_1, \ldots, g_n}$ be a basis for $\mathfrak g$ with corresponding dual basis $\set{x_1, \ldots, x_n}$.  We define $\Theta_{\mathfrak g}$ to be the operator
\[\Theta_{\mathfrak g} = x_1 \otimes g_1 + \cdots + x_n \otimes g_n.\]
As an element of $\mathfrak g^\ast \otimes_k \mathfrak g \simeq \hom_k(\mathfrak g, \mathfrak g)$ this is just the identity map and is therefore independent of the choice of basis.  Now $\Theta_{\mathfrak g}$ acts on $k[\mathcal N_p(\mathfrak g)] \otimes_k M \simeq k[\mathcal N_p(\mathfrak g)]^{\dim M}$ as a degree $1$ endomorphism of graded $k[\mathcal N_p(\mathfrak g)]$-modules (where $\deg x_i = 1$).  The map of sheaves corresponding to this homomorphism is the global operator.

\begin{Def}
Given a $\mathfrak g$-module $M$ we define $\widetilde M = \OPG \otimes_k M$.  The \emph{global operator} corresponding to $M$ is the sheaf map
\[\Theta_M\colon \widetilde M \to \widetilde M(1)\]
induced by the action of $\Theta_{\mathfrak g}$.
\end{Def}

\begin{Ex} \label{exV2glob}
We have $\Theta_{\slt} = x \otimes e + y \otimes f + z \otimes h$.  Consider the Weyl module $V(2)$ from \cref{exV2JT}.  The global operator corresponding to $V(2)$ is the sheaf map
\[\begin{bmatrix} 2z & 2x & 0 \\ y & 0 & x \\ 0 & 2y & -2z \end{bmatrix}\colon\OPG[\slt]^3 \to \OPG[\slt](1)^3.\]
Taking the pullback through the map $\iota\colon\mathbb P^1 \to \PG[\slt]$ from \cref{exNslt} we get that $\Theta_{V(2)}$ is the sheaf map
\[\begin{bmatrix} 2st & 2s^2 & 0 \\ -t^2 & 0 & s^2 \\ 0 & -2t^2 & -2st \end{bmatrix}\colon\OP^3 \to \OP(2)^3.\]
\end{Ex}

The global operator $\Theta_M$ is not an endomorphism but we may still compose it with itself if we shift the degree of successive copies.  Given $j \in \mathbb N$ we define
\begin{align*}
\gker[j]{M} &= \ker\left[\Theta_M(j-1)\circ\cdots\circ\Theta_M(1)\circ\Theta_M\right], \\
\gim[j]{M} &= \im\left[\Theta_M(-1)\circ\cdots\circ\Theta_M(1-j)\circ\Theta_M(-j)\right], \\
\gcoker[j]{M} &= \coker\left[\Theta_M(-1)\circ\cdots\circ\Theta_M(1-j)\circ\Theta_M(-j)\right],
\end{align*}
so that $\gker[j]{M}$ and $\gim[j]{M}$ are subsheafs of $\widetilde M$, and $\gcoker[j]{M}$ is a quotient of $\widetilde M$.

To see how these sheaves encode information about the Jordan type of $M$ recall that the $j$-rank of a partition $\lambda$ is the number of boxes in the Young diagram of $\lambda$ that are not contained in the first $j$ columns.  For example the $2$-rank of $[4]^2[2][1]$ (from \cref{exPart}) is $4$.  If one knows the $j$-rank of a partition $\lambda$ for all $j$, then one knows the size of each column in the Young diagram of $\lambda$ and can therefore recover $\lambda$.  Thus if one knows the local $j$-rank of a module $M$ for all $j$ then one knows its local Jordan type.

\begin{Def}
Let $M$ be a $\mathfrak g$-module and let $v \in \PG$.  Set $\rank^j(v, M)$ equal to the $j$-rank of the partition $\jtype(v, M)$.  The \emph{local $j$-rank} of $M$ is the function
\[\rank^j(-, M)\colon\PG \to \mathbb N_0\]
so defined.
\end{Def}

\begin{Thm}[{\cite[3.2]{starkHo1}}]
Let $M$ be a $\mathfrak g$-module and $U \subseteq \PG$ an open set.  The local $j$-rank is constant on $U$ if and only if the restriction $\gcoker[j]{M}|_U$ is a locally free sheaf.  When this is the case $\gker[j]{M}|_U$ and $\gim[j]{M}|_U$ are also locally free and $\rank^j(v, M) = \rank\gim[j]{M}$ for all $v \in U$.
\end{Thm}

We will also be interested in the sheaves $\F_i(M)$ for $1 \leq i \leq p$.  These were first defined by Benson and Pevtsova \cite{benPevtVectorBundles} for $kE$-modules where $E$ is an elementary abelian $p$-group.

\begin{Def}
Let $M$ be a $\mathfrak g$-module and $1 \leq i \leq p$ an integer.  Then
\[\F_i(M) = \frac{\gker{M} \cap \gim[i-1]{M}}{\gker{M} \cap \gim[i]{M}}.\]
\end{Def}

We end the section by stating two theorems which not only illustrate the utility of these sheaves but will be used in an essential way in \Cref{secLieEx} when calculating $\F_i(M)$ where $M$ is a Weyl module for $\slt$.  Both theorems were originally published by Benson and Pevtsova \cite{benPevtVectorBundles} but with minor errors.  These errors have been corrected in the given reference.

\begin{Thm}[{\cite[3.7]{starkHo1}}] \label{thmOm}
Let $M$ be a $\mathfrak g$-module and $1 \leq i < p$ an integer.  Then
\[\F_i(M) \simeq \F_{p-i}(\Omega M)(p-i)\]
\end{Thm}

\begin{Thm}[{\cite[3.8]{starkHo1}}] \label{thmFi}
Let $U \subseteq \PG$ be open.  The local Jordan type of a $\mathfrak g$-module $M$ is constant on $U$ if and only if the restrictions $\F_i(M)|_U$ are locally free for all $1 \leq i \leq p$.  When this is the case and $a_i = \rank\F_i(M)$ we have $\jtype(v, M) = [p]^{a_p}[p-1]^{a_{p-1}}\cdots[1]^{a_1}$ for all $v \in U$.
\end{Thm}

\section{The category of $\slt$-modules} \label{secSl2}

The calculations in \Cref{secLieEx} will be based on detailed information about the category of $\slt$-modules, which we develop in this section.  The indecomposable $\slt$-modules have been classified, each is one of the following four types: a Weyl module $V(\lambda)$, the dual of a Weyl module $V(\lambda)^\ast$, an indecomposable projective module $Q(\lambda)$, or a non-constant module $\Phi_\xi(\lambda)$.  Explicit bases for the first three types are known; we will remind the reader of these formulas and develop similar formulas for the $\Phi_\xi(\lambda)$.  We will also calculate the local Jordan type $\jtype(-, M)\colon\mathbb P^1 \to \mathscr P_p$ for each indecomposable $M$.  Finally we will calculate the Heller shifts $\Omega(V(\lambda))$.

We begin by stating the results for each of the four types and the classification.  Recall the standard basis for $\slt$ is $\set{e, f, h}$ where
\[e = \begin{bmatrix} 0 & 1 \\ 0 & 0 \end{bmatrix}, \qquad f = \begin{bmatrix} 0 & 0 \\ 1 & 0 \end{bmatrix}, \quad \text{and} \quad h = \begin{bmatrix} 1 & 0 \\ 0 & -1 \end{bmatrix}.\]
Let $\lambda$ be a non-negative integer and write $\lambda = rp + a$ where $0 \leq a < p$ is the remainder of $\lambda$ modulo $p$.  Each type is parametrized by the choice of $\lambda$, with the parametrization of $\Phi_\xi(\lambda)$ requiring also a choice of $\xi \in \mathbb P^1$.  The four types are as follows:

\begin{itemize}
\item The {\bf Weyl modules} $V(\lambda)$.
\begin{center}
\begin{tabular}{rrl}
Basis: & \multicolumn{2}{l}{$\set{v_0, v_1, \ldots, v_\lambda}$} \hspace{150pt} \\
Action: & $ev_i$ & \hspace{-7pt}$= (\lambda - i + 1)v_{i - 1}$ \\
& $fv_i$ & \hspace{-7pt}$= (i + 1)v_{i + 1}$ \\
& $hv_i$ & \hspace{-7pt}$= (\lambda - 2i)v_i$ \\
Graph: & \multicolumn{2}{l}{\Cref{figV}} \\
Local Jordan type: & \multicolumn{2}{l}{Constant Jordan type $[p]^r[a + 1]$}
\end{tabular}
\end{center}
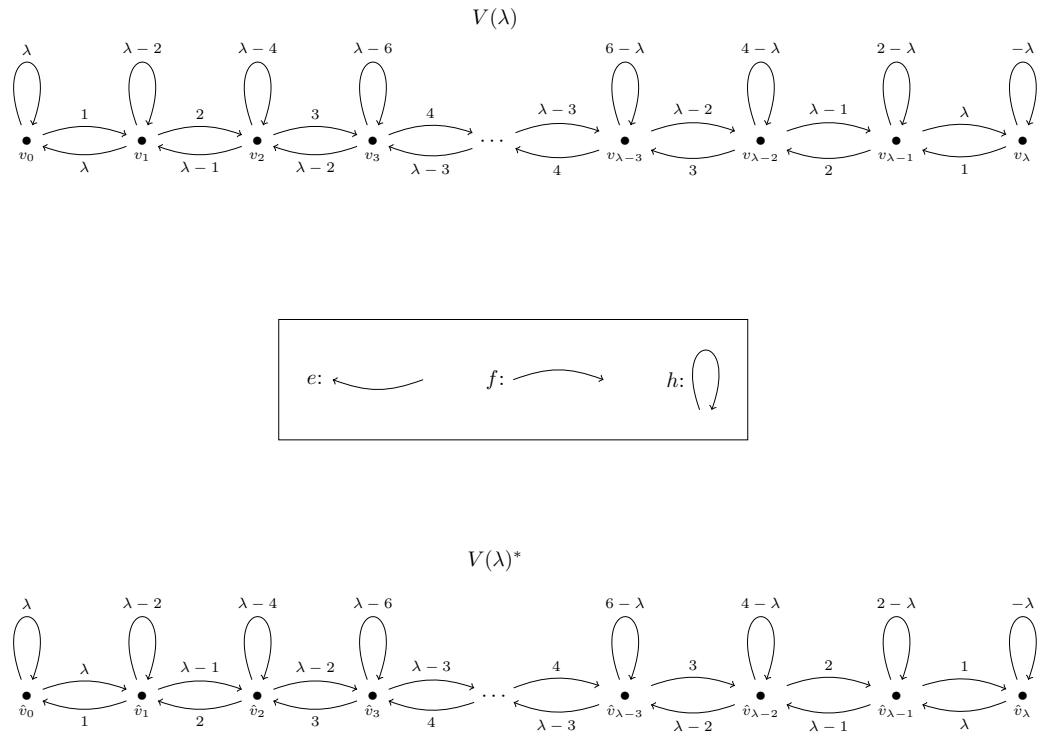
\begin{sidewaysfigure}[p]
\centering
\vspace*{350pt}
\begin{tikzpicture}	[description/.style={fill=white,inner sep=2pt}]
\useasboundingbox (-7, -5.5) rectangle (7, 4.2);
\scope[transform canvas={scale=.8}]
	\matrix (m) [matrix of math nodes, row sep=31pt,
	column sep=40pt, text height=1.5ex, text depth=0.25ex]
	{ \\ \\ \\ \underset{v_0}{\bullet} & \underset{v_1}{\bullet} & \underset{v_2}{\bullet} & \underset{v_3}{\bullet} & \cdots & \underset{v_{\lambda - 3}}{\bullet} & \underset{v_{\lambda - 2}}{\bullet} & \underset{v_{\lambda - 1}}{\bullet} & \underset{v_\lambda}{\bullet} \\ \\ \\ \\ \\ \\ \\ \\ \underset{\hat v_0}{\bullet} & \underset{\hat v_1}{\bullet} & \underset{\hat v_2}{\bullet} & \underset{\hat v_3}{\bullet} & \cdots & \underset{\hat v_{\lambda - 3}}{\bullet} & \underset{\hat v_{\lambda - 2}}{\bullet} & \underset{\hat v_{\lambda - 1}}{\bullet} & \underset{\hat v_\lambda}{\bullet} \\};
	\path[->,font=\scriptsize]
	(m-4-1) edge [bend left=20] node[auto] {$1$} (m-4-2)
	(m-4-2) edge [bend left=20] node[auto] {$\lambda$} (m-4-1)
			edge [bend left=20] node[auto] {$2$} (m-4-3)
	(m-4-3) edge [bend left=20] node[auto] {$\lambda - 1$} (m-4-2)
			edge [bend left=20] node[auto] {$3$} (m-4-4)
	(m-4-4) edge [bend left=20] node[auto] {$\lambda - 2$} (m-4-3)
			edge [bend left=20] node[auto] {$4$} (m-4-5)
	(m-4-5) edge [bend left=20] node[auto] {$\lambda - 3$} (m-4-4)
			edge [bend left=20] node[auto] {$\lambda - 3$} (m-4-6)
	(m-4-6) edge [bend left=20] node[auto] {$4$} (m-4-5)
			edge [bend left=20] node[auto] {$\lambda - 2$} (m-4-7)
	(m-4-7) edge [bend left=20] node[auto] {$3$} (m-4-6)
			edge [bend left=20] node[auto] {$\lambda - 1$} (m-4-8)
	(m-4-8) edge [bend left=20] node[auto] {$2$} (m-4-7)
			edge [bend left=20] node[auto] {$\lambda$} (m-4-9)
	(m-4-9) edge [bend left=20] node[auto] {$1$} (m-4-8);
	\draw[<-] (m-4-1) .. controls +(70:50pt) and +(110:50pt) .. node[pos=.5, above]{\scriptsize $\lambda$} (m-4-1);
	\draw[<-] (m-4-2) .. controls +(70:50pt) and +(110:50pt) .. node[pos=.5, above]{\scriptsize $\lambda - 2$} (m-4-2);
	\draw[<-] (m-4-3) .. controls +(70:50pt) and +(110:50pt) .. node[pos=.5, above]{\scriptsize $\lambda - 4$} (m-4-3);
	\draw[<-] (m-4-4) .. controls +(70:50pt) and +(110:50pt) .. node[pos=.5, above]{\scriptsize $\lambda - 6$} (m-4-4);
	\draw[<-] (m-4-6) .. controls +(70:50pt) and +(110:50pt) .. node[pos=.5, above]{\scriptsize $6 - \lambda$} (m-4-6);
	\draw[<-] (m-4-7) .. controls +(70:50pt) and +(110:50pt) .. node[pos=.5, above]{\scriptsize $4 - \lambda$} (m-4-7);
	\draw[<-] (m-4-8) .. controls +(70:50pt) and +(110:50pt) .. node[pos=.5, above]{\scriptsize $2 - \lambda$} (m-4-8);
	\draw[<-] (m-4-9) .. controls +(70:50pt) and +(110:50pt) .. node[pos=.5, above]{\scriptsize $-\lambda$} (m-4-9);
	\path[draw] (-4.1, -2) rectangle (3.7, 0);
	\draw (-3.5, -1) node {$e$:};
	\draw[<-] (-3.2, -1) .. controls +(-20:18pt) and +(200:18pt) .. (-1.7, -1);
	\draw (-.5, -1) node {$f$:};
	\draw[->] (-.2, -1) .. controls +(20:18pt) and +(160:18pt) .. (1.3, -1);
	\draw (2.5, -1) node {$h$:};
	\draw[<-] (3.1, -1.5) .. controls +(70:40pt) and +(110:40pt) ..  (2.9, -1.5);
	\draw (-.5, 5) node {$V(\lambda)$};
	\path[->,font=\scriptsize]
	(m-12-1) edge [bend left=20] node[auto] {$\lambda$} (m-12-2)
	(m-12-2) edge [bend left=20] node[auto] {$1$} (m-12-1)
			edge [bend left=20] node[auto] {$\lambda - 1$} (m-12-3)
	(m-12-3) edge [bend left=20] node[auto] {$2$} (m-12-2)
			edge [bend left=20] node[auto] {$\lambda - 2$} (m-12-4)
	(m-12-4) edge [bend left=20] node[auto] {$3$} (m-12-3)
			edge [bend left=20] node[auto] {$\lambda - 3$} (m-12-5)
	(m-12-5) edge [bend left=20] node[auto] {$4$} (m-12-4)
			edge [bend left=20] node[auto] {$4$} (m-12-6)
	(m-12-6) edge [bend left=20] node[auto] {$\lambda - 3$} (m-12-5)
			edge [bend left=20] node[auto] {$3$} (m-12-7)
	(m-12-7) edge [bend left=20] node[auto] {$\lambda - 2$} (m-12-6)
			edge [bend left=20] node[auto] {$2$} (m-12-8)
	(m-12-8) edge [bend left=20] node[auto] {$\lambda - 1$} (m-12-7)
			edge [bend left=20] node[auto] {$1$} (m-12-9)
	(m-12-9) edge [bend left=20] node[auto] {$\lambda$} (m-12-8);
	\draw[<-] (m-12-1) .. controls +(70:50pt) and +(110:50pt) .. node[pos=.5, above]{\scriptsize $\lambda$} (m-12-1);
	\draw[<-] (m-12-2) .. controls +(70:50pt) and +(110:50pt) .. node[pos=.5, above]{\scriptsize $\lambda - 2$} (m-12-2);
	\draw[<-] (m-12-3) .. controls +(70:50pt) and +(110:50pt) .. node[pos=.5, above]{\scriptsize $\lambda - 4$} (m-12-3);
	\draw[<-] (m-12-4) .. controls +(70:50pt) and +(110:50pt) .. node[pos=.5, above]{\scriptsize $\lambda - 6$} (m-12-4);
	\draw[<-] (m-12-6) .. controls +(70:50pt) and +(110:50pt) .. node[pos=.5, above]{\scriptsize $6 - \lambda$} (m-12-6);
	\draw[<-] (m-12-7) .. controls +(70:50pt) and +(110:50pt) .. node[pos=.5, above]{\scriptsize $4 - \lambda$} (m-12-7);
	\draw[<-] (m-12-8) .. controls +(70:50pt) and +(110:50pt) .. node[pos=.5, above]{\scriptsize $2 - \lambda$} (m-12-8);
	\draw[<-] (m-12-9) .. controls +(70:50pt) and +(110:50pt) .. node[pos=.5, above]{\scriptsize $-\lambda$} (m-12-9);
	\draw (-.5, -4) node {$V(\lambda)^\ast$};
\endscope
\end{tikzpicture}
\caption{Graphs of $V(\lambda)$ and $V(\lambda)^\ast$} \label{figV}
\end{sidewaysfigure}
\item The {\bf dual Weyl modules} $V(\lambda)^\ast$.
\begin{center}
\begin{tabular}{rrl}
Basis: & \multicolumn{2}{l}{$\set{\hat v_0, \hat v_1, \ldots, \hat v_\lambda}$} \hspace{150pt} \\
Action: & $e\hat v_i$ & \hspace{-7pt}$= i\hat v_{i - 1}$ \\
& $f\hat v_i$ & \hspace{-7pt}$= (\lambda - i)\hat v_{i + 1}$ \\
& $h\hat v_i$ & \hspace{-7pt}$= (\lambda - 2i)\hat v_i$ \\
Graph: & \multicolumn{2}{l}{\Cref{figV}} \\
Local Jordan type: & \multicolumn{2}{l}{Constant Jordan type $[p]^r[a + 1]$}
\end{tabular}
\end{center}
\item The {\bf projectives} $Q(\lambda)$.

Define $Q(p - 1) = V(p - 1)$.  For $0 \leq \lambda < p - 1$ we define $Q(\lambda)$ via
\begin{center}
\begin{tabular}{rrl}
Basis: & \multicolumn{2}{l}{$\set{v_0, v_1, \ldots, v_{2p - \lambda - 2}} \cup \set{w_{p - \lambda - 1}, w_{p - \lambda}, \ldots, w_{p - 1}}$} \hspace{0pt} \\
Action: & $ev_i$ & \hspace{-7pt}$= -(\lambda + i + 1)v_{i - 1}$ \\
& $fv_i$ & \hspace{-7pt}$= (i + 1)v_{i + 1}$ \\
& $hv_i$ & \hspace{-7pt}$= -(\lambda + 2i + 2)v_i$ \\
& $ew_i$ & \hspace{-7pt}$= -(\lambda + i + 1)w_{i - 1} + \frac{1}{i}v_{i - 1}$ \\
& $fw_i$ & \hspace{-7pt}$= (i + 1)w_{i + 1} - \frac{1}{\lambda + 1}\delta_{-1, i}v_p$ \\
& $hw_i$ & \hspace{-7pt}$= -(\lambda + 2i + 2)w_i$ \\
Graph: & \multicolumn{2}{l}{\Cref{figQ}} \\
Local Jordan type: & \multicolumn{2}{l}{Constant Jordan type $[p]^2$}
\end{tabular}
\end{center}
\begin{sidewaysfigure}[p]
\centering
\vspace*{350pt}
\begin{tikzpicture}	[description/.style={fill=white,inner sep=2pt}]
\useasboundingbox (-8.5, -5) rectangle (8.5, 3.5);
\scope[transform canvas={scale=.8}]
	\matrix (m) [matrix of math nodes, row sep=31pt,
	column sep=12pt, text height=1ex, text depth=0.25ex]
	{ &&&&& \underset{w_{p - \lambda - 1}}{\bullet} && \underset{w_{p - \lambda}}{\bullet} && \underset{w_{p - \lambda + 1}}{\bullet} && \cdots && \underset{w_{p - 3}}{\bullet} && \underset{w_{p - 2}}{\bullet} && \underset{w_{p - 1}}{\bullet} \\
	\underset{v_0}{\bullet} && \cdots && \underset{v_{p - \lambda - 2}}{\bullet} &&&&&&&&&&&&&& \underset{v_p}{\bullet} && \cdots && \underset{v_{2p - \lambda - 2}}{\bullet} \\
	&&&&& \underset{v_{p - \lambda - 1}}{\bullet} && \underset{v_{p - \lambda}}{\bullet} && \underset{v_{p - \lambda + 1}}{\bullet} && \cdots && \underset{v_{p - 3}}{\bullet} && \underset{v_{p - 2}}{\bullet} && \underset{v_{p - 1}}{\bullet} \\};
	\path[->,font=\scriptsize]
	(m-2-1)  edge [bend left=20] node[auto] {$1$} (m-2-3)
	(m-2-3)  edge [bend left=20, shorten >=-7pt] node[auto, xshift=18pt] {$-\lambda - 2$} (m-2-5)
	(m-1-6)  edge [bend left=20] node[auto, xshift=-10pt] {$-\lambda$} (m-1-8)
	(m-1-8)  edge [bend left=20] node[auto, xshift=15pt] {$1 - \lambda$} (m-1-10)
	(m-1-10) edge [bend left=20] node[auto, xshift=-15pt] {$2 - \lambda$} (m-1-12)
	(m-1-12) edge [bend left=20] node[auto] {$-3$} (m-1-14)
	(m-1-14) edge [bend left=20] node[auto] {$-2$} (m-1-16)
	(m-1-16) edge [bend left=20] node[auto] {$-1$} (m-1-18)
	(m-3-6)  edge [bend left=20] node[auto, xshift=-10pt] {$-\lambda$} (m-3-8)
	(m-3-8)  edge [bend left=20] node[auto, xshift=15pt] {$1 - \lambda$} (m-3-10)
	(m-3-10) edge [bend left=20] node[auto, xshift=-15pt] {$2 - \lambda$} (m-3-12)
	(m-3-12) edge [bend left=20] node[auto] {$-3$} (m-3-14)
	(m-3-14) edge [bend left=20] node[auto] {$-2$} (m-3-16)
	(m-3-16) edge [bend left=20] node[auto] {$-1$} (m-3-18)
	(m-2-3)  edge [bend left=20] node[auto] {$-\lambda - 2$} (m-2-1)
	(m-2-5)  edge [bend left=20] node[auto, xshift=9pt] {$1$} (m-2-3)
	(m-1-8)  edge [bend left=20] node[auto, xshift=-10pt] {$-1$} (m-1-6)
	(m-1-10) edge [bend left=20] node[auto, xshift=10pt] {$-2$} (m-1-8)
	(m-1-12) edge [bend left=20] node[auto, xshift=-12pt] {$-3$} (m-1-10)
	(m-1-14) edge [bend left=20] node[auto] {$2 - \lambda$} (m-1-12)
	(m-1-16) edge [bend left=20] node[auto] {$1 - \lambda$} (m-1-14)
	(m-1-18) edge [bend left=20] node[auto] {$-\lambda$} (m-1-16)
	(m-3-8)  edge [bend left=20] node[auto, xshift=-10pt] {$-1$} (m-3-6)
	(m-3-10) edge [bend left=20] node[auto, xshift=10pt] {$-2$} (m-3-8)
	(m-3-12) edge [bend left=20] node[auto, xshift=-12pt] {$-3$} (m-3-10)
	(m-3-14) edge [bend left=20] node[auto] {$2 - \lambda$} (m-3-12)
	(m-3-16) edge [bend left=20] node[auto] {$1 - \lambda$} (m-3-14)
	(m-3-18) edge [bend left=20] node[auto] {$-\lambda$} (m-3-16)
	(m-2-19) edge [bend left=20] node[auto] {$1$} (m-2-21)
	(m-2-21) edge [bend left=20, shorten >=-7pt] node[auto, xshift=18pt] {$-\lambda - 2$} (m-2-23)
	(m-2-23) edge [bend left=20] node[auto, xshift=8pt] {$1$} (m-2-21)
	(m-2-21) edge [bend left=20] node[auto] {$-\lambda - 2$} (m-2-19)
	(m-1-6)  edge[shorten <=7pt] node[below, xshift=5pt] {$\frac{-1}{\lambda + 1}$} (m-2-5)
	(m-2-5)  edge[shorten <=7pt] node[pos=.45, below, xshift=-10pt] {$-\lambda - 1$} (m-3-6)
	(m-1-18) edge[shorten <=5pt] node[below, xshift=-5pt] {$\frac{-1}{\lambda + 1}$} (m-2-19)
	(m-2-19) edge node[auto, xshift=-5pt] {$-\lambda - 1$} (m-3-18)
	(m-1-8)  edge[shorten <=5pt] node[pos=.6, above, xshift=-5pt] {$\frac{-1}{\lambda}$} (m-3-6)
	(m-1-10) edge[shorten <=6pt] node[pos=.6, above, xshift=-5pt] {$\frac{-1}{\lambda - 1}$} (m-3-8)
	(m-1-12) edge node[pos=.6, above, xshift=-5pt] {$\frac{-1}{\lambda - 2}$} (m-3-10)
	(m-1-14) edge[shorten <=5pt] node[pos=.6, above, xshift=-5pt] {$-\frac{1}{3}$} (m-3-12)
	(m-1-16) edge[shorten <=5pt] node[pos=.6, above, xshift=-5pt] {$-\frac{1}{2}$} (m-3-14)
	(m-1-18) edge[shorten <=5pt] node[pos=.6, above, xshift=-5pt] {$-1$} (m-3-16);
	\draw[<-] (m-2-1)  .. controls +(70:50pt) and +(110:50pt) .. node[pos=.5, above]{\scriptsize $-\lambda - 2$} (m-2-1);
	\draw[<-] (m-2-5)  .. controls +(70:50pt) and +(110:50pt) .. node[pos=.5, above]{\scriptsize $\lambda + 2$} (m-2-5);
	\draw[<-] (m-2-19)  .. controls +(70:50pt) and +(110:50pt) .. node[pos=.5, above]{\scriptsize $-\lambda - 2$} (m-2-19);
	\draw[<-] (m-2-23)  .. controls +(70:50pt) and +(110:50pt) .. node[pos=.5, above]{\scriptsize $\lambda + 2$} (m-2-23);
	\draw[<-] (m-1-6)  .. controls +(70:50pt) and +(110:50pt) .. node[pos=.5, above]{\scriptsize $\lambda$} (m-1-6);
	\draw[<-] (m-1-8)  .. controls +(70:50pt) and +(110:50pt) .. node[pos=.5, above]{\scriptsize $\lambda - 2$} (m-1-8);
	\draw[<-] (m-1-10) .. controls +(70:50pt) and +(110:50pt) .. node[pos=.5, above]{\scriptsize $\lambda - 4$} (m-1-10);
	\draw[<-] (m-1-14) .. controls +(70:50pt) and +(110:50pt) .. node[pos=.5, above]{\scriptsize $4 - \lambda$} (m-1-14);
	\draw[<-] (m-1-16) .. controls +(70:50pt) and +(110:50pt) .. node[pos=.5, above]{\scriptsize $2 - \lambda$} (m-1-16);
	\draw[<-] (m-1-18) .. controls +(70:50pt) and +(110:50pt) .. node[pos=.5, above]{\scriptsize $-\lambda$} (m-1-18);
	\draw[shorten >=5pt, shorten <=5pt, <-] (m-3-6)  .. controls +(250:50pt) and +(290:50pt) .. node[pos=.5, below]{\scriptsize $\lambda$} (m-3-6);
	\draw[shorten >=5pt, shorten <=5pt, <-] (m-3-8)  .. controls +(250:50pt) and +(290:50pt) .. node[pos=.5, below]{\scriptsize $\lambda - 2$} (m-3-8);
	\draw[shorten >=5pt, shorten <=5pt, <-] (m-3-10) .. controls +(250:50pt) and +(290:50pt) .. node[pos=.5, below]{\scriptsize $\lambda - 4$} (m-3-10);
	\draw[shorten >=5pt, shorten <=5pt, <-] (m-3-14) .. controls +(250:50pt) and +(290:50pt) .. node[pos=.5, below]{\scriptsize $4 - \lambda$} (m-3-14);
	\draw[shorten >=5pt, shorten <=5pt, <-] (m-3-16) .. controls +(250:50pt) and +(290:50pt) .. node[pos=.5, below]{\scriptsize $2 - \lambda$} (m-3-16);
	\draw[shorten >=5pt, shorten <=5pt, <-] (m-3-18) .. controls +(250:50pt) and +(290:50pt) .. node[pos=.5, below]{\scriptsize $-\lambda$} (m-3-18);
	\path[draw] (-5.7, -6) rectangle (5.9, -4);
	\draw (-5.1, -5) node {$e$:};
	\draw[<-] (-4.8, -5) .. controls +(-20:18pt) and +(200:18pt) .. (-3.3, -5);
	\draw (-2.9, -5) node {$+$};
	\draw[<-] (-2.8, -5.7) -- (-1.8, -4.3);
	\draw (-.9, -5) node {$f$:};
	\draw[->] (-.6, -5) .. controls +(20:18pt) and +(160:18pt) .. (1.3, -5);
	\draw (1.7, -5) node {$+$};
	\draw[->] (1.8, -4.3) -- (2.8, -5.7);
	\draw (3.5, -5) node {$h$:};
	\draw[<-] (4.2, -5.5) .. controls +(70:40pt) and +(110:40pt) ..  (4, -5.5);
	\draw (4.7, -5) node {$+$};
	\draw[<-] (5.2, -4.5) .. controls +(250:40pt) and +(290:40pt) ..  (5.4, -4.5);
	\draw (.25, 4) node {$Q(\lambda)$};
\endscope
\end{tikzpicture}
\caption{Graph of $Q(\lambda)$} \label{figQ}
\end{sidewaysfigure}
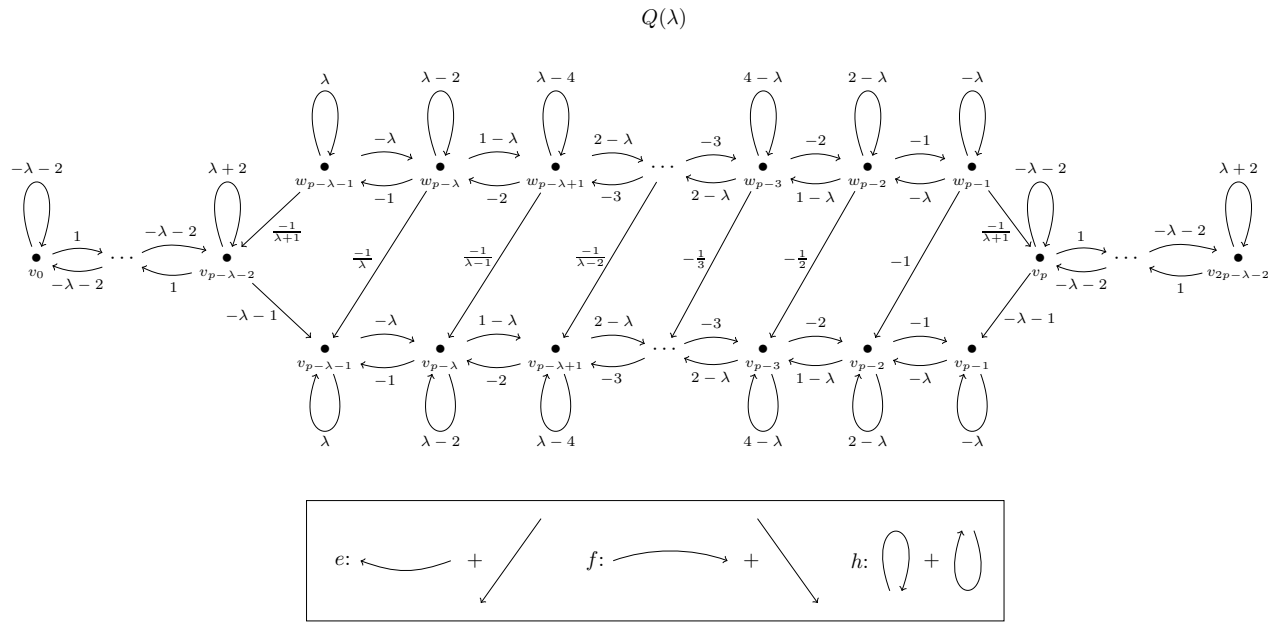

\item The {\bf non-constant modules} $\Phi_\xi(\lambda)$.

Assume $\lambda \geq p$ and let $\xi \in \mathbb P^1$.  If $\xi = [1:\varepsilon]$ then $\Phi_\xi(\lambda)$ is defined by
\begin{center}
\begin{tabular}{rrl}
Basis: & \multicolumn{2}{l}{$\set{w_{a + 1}, w_{a + 2}, \ldots, w_\lambda}$} \hspace{122pt} \\
Action: & $ew_i$ & \hspace{-7pt}$= (i + 1)(w_{i + 1} - {d \choose i}\varepsilon^{i - a}\delta_{\lambda, i}w_{a + 1})$ \\
& $fw_i$ & \hspace{-7pt}$= (\lambda - i + 1)w_{i - 1}$ \\
& $hw_i$ & \hspace{-7pt}$= (2i - \lambda)w_i$ \\
Graph: & \multicolumn{2}{l}{\Cref{figPhi}} \\
Local Jordan type: & \multicolumn{2}{l}{$[p]^{r-1}[p - a - 1][a + 1]$ at $\xi$ and $[p]^r$ elsewhere}
\end{tabular}
\end{center}
If $\xi = [0:1]$ then $\Phi_\xi(\lambda)$ is defined to be the submodule of $V(\lambda)$ spanned by the basis elements $\set{v_{a + 1}, v_{a + 2}, \ldots, v_\lambda}$.  It is also depicted in \Cref{figPhi} and has the same local Jordan type as above.
\begin{sidewaysfigure}[p]
\centering
\vspace*{350pt}
\begin{tikzpicture}	[description/.style={fill=white,inner sep=2pt}]
\useasboundingbox (-9, -5.7) rectangle (8, 4.8);
\scope[transform canvas={scale=.8}]
	\matrix (m) [matrix of math nodes, row sep=50pt,
	column sep=20pt, text height=1ex, text depth=0.25ex]
	{ & {} & {} &&&&&& {} &&& \\ \underset{w_\lambda}{\bullet} & \cdots & \underset{w_{qp + a + 1}}{\bullet} & \underset{w_{qp + a}}{\bullet} & \cdots & \underset{w_{qp}}{\bullet} & \underset{w_{qp - 1}}{\bullet} & \cdots & \underset{w_{(q - 1)p + a + 1}}{\bullet} & \underset{w_{(q - 1)p + a}}{\bullet} & \cdots & \underset{w_{a + 1}}{\bullet} \\ \\ \\ \\ \\ \\ && \underset{v_{a + 1}}{\bullet} & \underset{v_{a + 2}}{\bullet} & \underset{v_{a + 3}}{\bullet} & \underset{v_{a + 4}}{\bullet} & \cdots & \underset{v_{\lambda - 2}}{\bullet} & \underset{v_{\lambda - 1}}{\bullet} & \underset{v_{\lambda}}{\bullet} \\};
	\path[->,font=\scriptsize]
	(m-2-12) edge [bend left=20] node[auto] {$a + 2$} (m-2-11)
	(m-2-11) edge [bend left=20, shorten >=3pt] node[auto, xshift=-5pt] {$a$} (m-2-10)
	(m-2-10) edge [bend left=20, shorten >=7pt, shorten <=3pt] node[auto] {$a + 1$} (m-2-9)
	(m-2-9)  edge [bend left=20, shorten <=7pt] node[auto, xshift=7pt] {$a + 2$} (m-2-8)
	(m-2-8)  edge [bend left=20] node[auto, xshift=-10pt] {$-1$} (m-2-7)
	(m-2-6)  edge [bend left=20] node[auto] {$1$} (m-2-5)
	(m-2-5)  edge [bend left=20] node[auto, xshift=-5pt] {$a$} (m-2-4)
	(m-2-4)  edge [bend left=20] node[auto] {$a + 1$} (m-2-3)
	(m-2-3)  edge [bend left=20] node[auto, xshift=13pt] {$a + 2$} (m-2-2)
	(m-2-2)  edge [bend left=20] node[auto] {$a$} (m-2-1)
	(m-2-1)  edge [bend left=20] node[auto] {$1$} (m-2-2)
	(m-2-2)  edge [bend left=20, shorten >=-10pt] node[auto, xshift=12pt] {$-1$} (m-2-3)
	(m-2-4)  edge [bend left=20, shorten <=-7pt] node[auto, xshift=-7pt] {$1$} (m-2-5)
	(m-2-5)  edge [bend left=20, shorten >=-3pt] node[auto] {$a$} (m-2-6)
	(m-2-6)  edge [bend left=20, shorten <=-3pt, shorten >=-7pt] node[auto, xshift=13pt] {$a + 1$} (m-2-7)
	(m-2-7)  edge [bend left=20, shorten <=-7pt] node[auto, xshift=-15pt] {$a + 2$} (m-2-8)
	(m-2-8)  edge [bend left=20, shorten >=-10pt] node[auto, xshift=12pt] {$-1$} (m-2-9)
	(m-2-10) edge [bend left=20, shorten <=-10pt] node[auto, xshift=-10pt] {$1$} (m-2-11)
	(m-2-11) edge [bend left=20] node[auto] {$-1$} (m-2-12);
	\draw[<-] (m-2-12)  .. controls +(70:50pt) and +(110:50pt) .. node[pos=.5, above]{\scriptsize $a + 2$} (m-2-12);
	\draw[<-] (m-2-10)  .. controls +(70:50pt) and +(110:50pt) .. node[pos=.5, above]{\scriptsize $a$} (m-2-10);
	\draw[<-] (m-2-9) .. controls +(70:50pt) and +(110:50pt) .. node[pos=.5, above]{\scriptsize $a + 2$} (m-2-9);
	\draw[<-] (m-2-7) .. controls +(70:50pt) and +(110:50pt) .. node[pos=.5, above]{\scriptsize $-(a + 2)$} (m-2-7);
	\draw[<-] (m-2-6) .. controls +(70:50pt) and +(110:50pt) .. node[pos=.5, above]{\scriptsize $-a$} (m-2-6);
	\draw[<-] (m-2-4) .. controls +(70:50pt) and +(110:50pt) .. node[pos=.5, above]{\scriptsize $a$} (m-2-4);
	\draw[<-] (m-2-3) .. controls +(70:50pt) and +(110:50pt) .. node[pos=.5, above]{\scriptsize $a + 2$} (m-2-3);
	\draw[<-] (m-2-1) .. controls +(70:50pt) and +(110:50pt) .. node[pos=.5, above]{\scriptsize $a$} (m-2-1);
	\draw[shorten >=5pt, ->] (m-2-1) .. controls +(210:170pt) and +(250:150pt) .. node[pos=.21, below, xshift=-20pt]{\scriptsize $-(a + 1)\varepsilon^{rp}$} (m-2-12);
	\draw[shorten >=5pt, shorten <=5pt, ->] (m-2-4) .. controls +(220:130pt) and +(250:120pt) .. node[pos=.2, below, xshift=-30pt]{\scriptsize $-(a + 1)\binom{r}{q}\varepsilon^{qp}$} (m-2-12);
	\draw[shorten >=5pt, shorten <=8pt, ->] (m-2-10) .. controls +(220:70pt) and +(250:70pt) .. node[pos=.3, below, xshift=-43pt]{\scriptsize $-(a + 1)\binom{r}{q - 1}\varepsilon^{(q - 1)p}$} (m-2-12);
	\path[draw] (-6.2, -2.5) rectangle (4.9, -.5);
	\draw (-5.6, -1.5) node {$e$:};
	\draw[<-] (-5.3, -1.5) .. controls +(-20:18pt) and +(200:18pt) .. (-3.8, -1.5);
	\draw (-3.4, -1.5) node {$+$};
	\draw[->] (-2.8, -1.2) .. controls +(230:40pt) and +(250:30pt) .. (-1.8, -1.2);
	\draw (-.3, -1.5) node {$f$:};
	\draw[->] (0, -1.5) .. controls +(20:18pt) and +(160:18pt) .. (1.9, -1.5);
	\draw (3.5, -1.5) node {$h$:};
	\draw[<-] (4.2, -2) .. controls +(70:40pt) and +(110:40pt) ..  (4, -2);
	\draw (-.5, 7) node {$\Phi_{[1:\varepsilon]}(\lambda)$};
	\draw (.4, -4) node {$\Phi_{[0:1]}(\lambda)$};
	\path[->,font=\scriptsize]
	(m-8-3)  edge [bend left=20] node[auto] {$a + 2$} (m-8-4)
	(m-8-4)  edge [bend left=20] node[auto] {$-1$} (m-8-3)
			 edge [bend left=20] node[auto] {$a + 3$} (m-8-5)
	(m-8-5)  edge [bend left=20] node[auto] {$-2$} (m-8-4)
			 edge [bend left=20] node[auto] {$a + 4$} (m-8-6)
	(m-8-6)  edge [bend left=20] node[auto] {$-3$} (m-8-5)
			 edge [bend left=20] node[auto] {$a + 5$} (m-8-7)
	(m-8-7)  edge [bend left=20] node[auto] {$-4$} (m-8-6)
			 edge [bend left=20] node[auto] {$\lambda - 2$} (m-8-8)
	(m-8-8)  edge [bend left=20] node[auto] {$3$} (m-8-7)
			 edge [bend left=20] node[auto] {$\lambda - 1$} (m-8-9)
	(m-8-9)  edge [bend left=20] node[auto] {$2$} (m-8-8)
			 edge [bend left=20] node[auto] {$\lambda$} (m-8-10)
	(m-8-10) edge [bend left=20] node[auto] {$1$} (m-8-9);
	\draw[<-] (m-8-3) .. controls +(70:50pt) and +(110:50pt) .. node[pos=.5, above]{\scriptsize $-(a + 2)$} (m-8-3);
	\draw[<-] (m-8-4) .. controls +(70:50pt) and +(110:50pt) .. node[pos=.5, above]{\scriptsize $-(a + 4)$} (m-8-4);
	\draw[<-] (m-8-5) .. controls +(70:50pt) and +(110:50pt) .. node[pos=.5, above]{\scriptsize $-(a + 6)$} (m-8-5);
	\draw[<-] (m-8-6) .. controls +(70:50pt) and +(110:50pt) .. node[pos=.5, above]{\scriptsize $-(a + 8)$} (m-8-6);
	\draw[<-] (m-8-8) .. controls +(70:50pt) and +(110:50pt) .. node[pos=.5, above]{\scriptsize $-(a - 4)$} (m-8-8);
	\draw[<-] (m-8-9) .. controls +(70:50pt) and +(110:50pt) .. node[pos=.5, above]{\scriptsize $-(a - 2)$} (m-8-9);
	\draw[<-] (m-8-10) .. controls +(70:50pt) and +(110:50pt) .. node[pos=.5, above]{\scriptsize $-a$} (m-8-10);
\endscope
\end{tikzpicture}
\caption{Graphs of $\Phi_\xi(\lambda)$} \label{figPhi}
\end{sidewaysfigure}
\end{itemize}

\begin{Thm}[\cite{premetSl2}] \label{thmPremet}
Each of the following modules are indecomposable:
\begin{itemize}
\item $V(\lambda)$ and $Q(\lambda)$ for $0 \leq \lambda < p$.
\item $V(\lambda)$ and $V(\lambda)^\ast$ for $\lambda \geq p$ such that $p \nmid \lambda + 1$.
\item $\Phi_\xi(\lambda)$ for $\xi \in \mathbb P^1$ and $\lambda \geq p$ such that $p \nmid \lambda + 1$.
\end{itemize}
Moreover, these modules are pairwise non-isomorphic, save $Q(p-1) = V(p-1)$, and give a complete classification of the indecomposable restricted $\slt$-modules.
\end{Thm}

As stated before the explicit bases for $V(\lambda)$, $V(\lambda)^\ast$, and $Q(\lambda)$ are known; see, for example, Benkart and Osborn \cite{benkartSl2reps}.  For the local Jordan type of $V(\lambda)$ and $V(\lambda)^\ast$ the matrix that describes the action of $e$ with respect to the given basis is almost in Jordan normal form, one needs only to scale the basis elements appropriately, so we can immediately read off the local Jordan type at the point $ke \in \PG[\slt]$, and \cref{thmRatCJT} gives that these modules have constant Jordan type.  As \cref{thmProjCJT} gives the local Jordan type of the $Q(\lambda)$ all that is left is to justify the explicit description of $\Phi_\xi(\lambda)$ and its local Jordan type.

First we recall the definition of $\Phi_\xi(\lambda)$.  Let $V = k^2$ be the standard representation of $\SL_2$, then the dual $V^\ast$ is a $2$-dimensional representation with basis $\set{x, y}$ (dual to the standard basis for $V$).  The induced representation on the symmetric product $S(V^\ast)$ is degree preserving and the dual $S^\lambda(V^\ast)^\ast$ of the degree $\lambda$ subspace is the Weyl module $V(\lambda)$.  Specifically, we let $v_i \in V(\lambda)$ be dual to $x^{\lambda - i}y^i$.

Let $B_2 \subseteq \SL_2$ be the Borel subgroup of upper triangular matrices and recall that the homogeneous space $\SL_2/B_2$ is isomorphic to $\mathbb P^1$ as a variety; the map $\phi\colon \mathbb P^1 \to \SL_2$ given by
\[[1:\varepsilon] \mapsto \begin{bmatrix} 0 & 1 \\ -1 & -\varepsilon \end{bmatrix} \qquad \text{ and } \qquad [0:1] \mapsto \begin{bmatrix} 1 & 0 \\ 0 & 1 \end{bmatrix},\]
factors to an explicit isomorphism $\mathbb P^1 \overset{\sim}{\to} \SL_2/B_2$.

\begin{Def}[\cite{premetSl2}]
Let $\Phi(\lambda)$ be the $\slt$-submodule of $V(\lambda)$ spanned by the vectors $\set{v_{a + 1}, v_{a + 2}, \ldots, v_\lambda}$.  Given $\xi \in \mathbb P^1$ we define $\Phi_\xi(\lambda)$ to be the $\slt$-module $\phi(\xi)\Phi(\lambda)$.
\end{Def}

Observe first that $\Phi_{[0:1]}(\lambda) = \Phi(\lambda)$ so in this case we have the desired description.  Now assume $\xi = [1:\varepsilon]$ where $\varepsilon \in k$.  As $\phi(\xi)$ is invertible multiplication by it is an isomorphism so $\Phi_\xi(\lambda)$ has basis $\set{\phi(\xi)v_i}$.  Our basis for $\Phi_\xi(\lambda)$ will be obtained by essentially a row reduction of this basis, so to proceed we now compute the action of $\SL_2$ on $V(\lambda)$.  Observe:
\begin{align*}
\left(\begin{bmatrix} a & b \\ c & d \end{bmatrix}v_i\right)(x^{\lambda - j}y^j) &= v_i\left(\begin{bmatrix} d & -b \\ -c & a \end{bmatrix}x^{\lambda - j}y^j\right) \\
&= v_i\left(\sum_{s = 0}^{\lambda - j}\sum_{t = 0}^j\binom{\lambda - j}{s}\binom{j}{t}a^kb^{\lambda - j - k}c^td^{j - t}x^{s + t}y^{\lambda - k - s}\right) \\
&= \sum\binom{\lambda - j}{s}\binom{j}{t}a^sb^{\lambda - j - s}c^td^{j - t}
\end{align*}
where the sum is over pairs $(s, t) \in \mathbb N^2$ such that $0 \leq s \leq \lambda - j$, $0 \leq t \leq j$, and $s + t = \lambda - i$.  Such pairs come in the form $(\lambda - i - t, t)$ where $t$ ranges from $\max(0, j - i)$ to $\min(j, \lambda - i)$ therefore
\[\begin{bmatrix} a & b \\ c & d \end{bmatrix}v_i = \sum_{j = 0}^r\sum_{t = \max(0, j - i)}^{\min(j, \lambda - i)}\binom{\lambda - j}{\lambda - i - t}\binom{j}{t}a^{\lambda - i - t}b^{t + i - j}c^td^{j - t}v_j.\]
For computing $\Phi_\xi(\lambda)$ we will need only the following special case:
\[\phi(\xi)v_i = \begin{bmatrix} 0 & 1 \\ -1 & -\varepsilon \end{bmatrix}v_i = \sum_{j = \lambda - i}^\lambda(-1)^j\binom{j}{\lambda - i}\varepsilon^{i + j - \lambda}v_j.\]

\begin{Prop} \label{propBas}
Given $i = qp + b$, $0 \leq b < p$, define
\[w_i = \begin{cases} v_{\lambda - i} - \binom{r}{q}\varepsilon^{qp}v_{\lambda - b} & \text{if} \ \ b \leq a \\ v_{\lambda - i} & \text{if} \ \ b  > a. \end{cases}\]
Then the vectors $w_{a + 1}, w_{a + 2}, \ldots, w_\lambda$ form a basis of $\Phi_\xi(\lambda)$.
\end{Prop}
\begin{proof}
We will prove by induction that that for all $a + 1 \leq i \leq \lambda$ we have $\spn_k\set{\phi(\xi)v_{a + 1}, \ldots, \phi(\xi)v_i} = \spn_k\set{w_{a + 1}, \ldots, w_i}$.  For the base case the formula above gives
\begin{align*}
\phi(\xi)v_{a + 1} &= \sum_{j = rp - 1}^\lambda(-1)^j\binom{j}{rp - 1}\varepsilon^{j - rp + 1}v_j \\
&= (-1)^{rp - 1}\binom{rp - 1}{rp - 1}v_{rp - 1} \\
&= (-1)^{rp - 1}w_{a + 1}
\end{align*}
so clearly the statement is true.

For the inductive step we assume the statement holds for integers less than $i$.  Then by hypothesis we have
\[\spn_k\set{\phi(\xi)v_{a + 1}, \ldots, \phi(\xi)v_i} = \spn_k\set{w_{a + 1}, \ldots, w_{i - 1}, \phi(\xi)v_i}\]
and can replace $\phi(\xi)v_i$ with the vector
\[w' = (-1)^{\lambda - i}\phi(\xi)v_i - \sum_{j = a + 1}^{i - 1}(-1)^{i - j}\binom{\lambda - j}{\lambda - i}\varepsilon^{i - j}w_j\]
to get another spanning set.  We then show that $w' = w_i$ by checking that the coordinates of each vector are equal.  Note that for $j < \lambda - i$ the coefficient of $v_j$ in each of the factors of $w'$ is zero, as it is in $w_i$.  The coefficient of $v_{\lambda - i}$ in $(-1)^{\lambda - i}\phi(\xi)v_i$ is $1$ and in each $w_j$, $a + 1 \leq j < i$ it is zero, hence the coefficient in $w'$ is $1$, as it is in $w_i$.

Next assume $\lambda - i < j < rp$.  Then only $\phi(\xi)v_i$ and $w_{\lambda - j}$ contribute a $v_j$ term therefore the coefficient of $v_j$ in $w'$ is
\[(-1)^{\lambda - i + j}\binom{j}{\lambda - i}\varepsilon^{i + j - \lambda} - (-1)^{i + j - \lambda}\binom{j}{\lambda - i}\varepsilon^{i + j - \lambda} = 0\]
which again agrees with $w_i$.  Now all that's left is to check the coefficients of $v_{rp}, v_{rp + 1}, \ldots, v_\lambda$.  Note that $w_t$ has a $v_{rp + j}$ term only if
\[t = p + a - j, 2p + a - j, \ldots, \lambda - j\]
and the coefficient of $v_{rp + j}$ in $w_{tp + a - j}$, for $1 \leq t \leq r$, is
\[-\binom{r}{t}\varepsilon^{tp}.\]
Thus the coefficient of $v_{rp + j}$ in $w'$ is
\[(-1)^{a - i - j}\binom{rp + j}{\lambda - i}\varepsilon^{i + j - a} + \sum(-1)^{i + j - tp - a}\binom{r}{t}\binom{(r - t)p + j}{\lambda - i}\varepsilon^{i + j - a}\]
where the sum is over those $t$ such that $1 \leq t \leq r$ and $tp + a \leq i + j - 1$.

From here there are several cases.  Assume first that $a < b$.  Then, from $b < p$ we get $p + a - b > a \geq j$ hence any binomial coefficient those top number equals $a - b$ in $k$ and whose bottom number equals $j$ in $k$ will be zero because there will be a carry.  Both $\binom{\lambda - t}{\lambda - i}$ and $\binom{(r - t)p + j}{\lambda - i}$ satisfy this condition therefore the coefficient of $v_{rp + j}$ in $w'$ is zero.  Thus if $a < b$ then we have $w' = w_i$.

Next assume that $b \leq a$.  Then the formula above for the coefficient of $v_{rp + j}$ in $w'$ becomes
\begin{align*}
&(-1)^{a - i - j}\left[\binom{r}{q} + \sum(-1)^{tp}\binom{r}{t}\binom{r - t}{r - q}\right]\binom{j}{a - b}\varepsilon^{i + j - a} \\
&\qquad = (-1)^{a - i - j}\left[\binom{r}{q} + \sum(-1)^t\binom{r}{q}\binom{q}{t}\right]\binom{j}{a - b}\varepsilon^{i + j - a} \\
&\qquad = (-1)^{a - i - j}\left[1 + \sum(-1)^t\binom{q}{t}\right]\binom{r}{q}\binom{j}{a - b}\varepsilon^{i + j - a}
\end{align*}
where the sum is over the same $t$ from above.  If $j < a - b$ then clearly this is zero.  If $j > a - b$ then the sum is over $t = 1, 2, \ldots, q$ and
\[1 + \sum_{t = 1}^q(-1)^t\binom{q}{t} = \sum_{t = 0}^q(-1)^t\binom{q}{t} = 0\]
so the only $v_{\lambda - b}$ occurs as a term in $w'$.  In that case the sum is over $t = 1, 2, \ldots, q - 1$ and
\[1 + \sum_{t = 1}^{q - 1}(-1)^t\binom{q}{t} = (-1)^{q + 1} + \sum_{t = 0}^q(-1)^t\binom{q}{t} = (-1)^{q + 1}\]
so the coefficient is
\[(-1)^{a - i - (a - b) + q + 1}\binom{r}{q}\varepsilon^{i + (a - b) - a} = -\binom{r}{q}\varepsilon^{qp}\]
as desired.  Thus $w' = w_i$ and the proof is complete.
\end{proof}

Now that we have a basis it's trivial to determine the action.
\begin{Prop}
Let $i = qp + b$, $a + 1 \leq i \leq \lambda$.  Then
\begin{align*}
ew_i &= \begin{cases} (i + 1)\left(w_{i + 1} - \binom{\lambda}{i}\varepsilon^{qp}w_{b + 1}\right) & \text{if} \ \ a = b \\ (i + 1)w_{i + 1} & \text{if} \ \ a \neq b \end{cases} \\
fw_i &= (\lambda - i + 1)w_{i - 1} \\
hw_i &= (2i - \lambda)w_i
\end{align*}
where $w_a = w_{\lambda + 1} = 0$.
\end{Prop}
\begin{proof}
The proof is just a case by case analysis.  We start with $e \in \slt$.  If $b < a$ then
\[ew_i = ev_{\lambda - i} - \binom{r}{q}\varepsilon^{qp}ev_{\lambda - b} = (i + 1)v_{\lambda - i - 1} - (b + 1)\binom{r}{q}\varepsilon^{qp}v_{\lambda - b - 1} = (i + 1)w_{i + 1}.\]
If $b = a$ then
\[ew_i = (i + 1)v_{\lambda - i - 1} - (b + 1)\binom{r}{q}\varepsilon^{qp}v_{\lambda - b - 1} = (i + 1)\left(w_{i + 1} - \binom{\lambda}{i}\varepsilon^{qp}w_{b + 1}\right).\]
If $p - 1 > b > a$ then
\[ew_i = ev_{\lambda - i} = (i + 1)v_{\lambda - i - 1} = (i + 1)w_{i + 1}.\]
Finally if $b = p - 1$ then
\[ew_i = (i + 1)v_{\lambda - i - 1} = 0.\]
All the above cases fit the given formula so we are done with $e$.  Next consider $f \in \slt$.  If $0 < b \leq a$ then
\begin{align*}
fw_i &= fv_{\lambda - i} - \binom{r}{q}\varepsilon^{qp}fv_{\lambda - b} \\
&= (\lambda - i + 1)v_{\lambda - i + 1} - (\lambda - b + 1)\binom{r}{q}\varepsilon^{qp}v_{\lambda - b + 1} \\
&= (\lambda - i + 1)w_{i - 1}.
\end{align*}
If $b = 0$ then
\[fw_i = fv_{\lambda - i} - \binom{r}{q}\varepsilon^{qp}fv_\lambda = (\lambda + 1)v_{\lambda - i + 1} = (\lambda - i + 1)w_{i - 1}.\]
If $a + 1 = b$ then
\[fw_i = fv_{\lambda - i} = (\lambda - i + 1)v_{\lambda - i + 1} = spv_{\lambda - i + 1} = 0.\]
Finally if $b > a + 1$ then
\[fw_i = (\lambda - i + 1)v_{\lambda - i + 1} = (\lambda - i + 1)w_{i - 1}.\]
All the above cases fit the given formula so we are done with $f$. Last but not least consider $h \in \slt$.  If $b \leq a$ then
\[hw_i = hv_{\lambda - i} - \binom{r}{q}\varepsilon^{qp}hv_{\lambda - b} = (2i - \lambda)v_{\lambda - i} - (2b - \lambda)\binom{r}{q}\varepsilon^{qp}v_{\lambda - b} = (2i - \lambda)w_i.\]
If $b > a$ then
\[hw_i = hv_{\lambda - i} = (2i - \lambda)v_{\lambda - i} = (2i - \lambda)w_i.\]
\end{proof}

Lastly we calculate that the Jordan type is as stated: $[p]^{r-1}[p - a - 1][a + 1]$ at $\xi$ and $[p]^r$ elsewhere.  First note that the result holds for $\Phi_{[0:1]}(\lambda)$ by \cref{lemBJType}; furthermore, that the point $[0:1] \in \mathbb P^1$ at which the Jordan type is $[p]^{r - 1}[p - a - 1][a + 1]$ corresponds to the line through $f \in \mathcal N_p(\slt)$ under the map
\begin{align*}
\iota\colon\mathbb P^1 &\to \PG[\slt] \\
[s : t] &\mapsto \begin{bmatrix} st & s^2 \\ -t^2 & -st \end{bmatrix}
\end{align*}
from \cref{exNslt}.  Let
\[\ad\colon\SL_2 \to \End(\slt)\]
be the adjoint action of $\SL_2$ on $\slt$.  As $V(\lambda)$ is a rational $\SL_2$-module this satisfies
\[\ad(g)(E)\cdot m = g\cdot(E\cdot(g^{-1}\cdot m))\]
for all $g \in \SL_2$, $E \in \slt$, and $m \in V(\lambda)$.  Along with $\Phi_{\xi}(\lambda) = \phi(\xi)\Phi_{[0:1]}(\lambda)$ this implies commutativity of the following diagram:
\[\xymatrix@C=50pt{\Phi_{[0:1]}(\lambda) \ar[r]^-{\phi(\xi)} \ar[d]_{\hspace*{40pt}E} & \Phi_\xi(\lambda) \ar[d]^{\ad(\phi(\xi))(E)} \\ \Phi_{[0:1]}(\lambda) \ar[r]_-{\phi(\xi)} & \Phi_\xi(\lambda)}\]
As multiplication by $\phi(\xi)$ is an isomorphism, letting $E$ range over $\mathcal N_p(\slt)$ we see that the module $\Phi_\xi(\lambda)$ has Jordan type $[p]^{r - 1}[p - a - 1][a + 1]$ at $\ad(\phi(\xi))(f)$ and $[p]^r$ elsewhere.  Then we simply calculate
\begin{align*}
\ad(\phi(\xi))(f) &= \begin{bmatrix} 0 & 1 \\ -1 & -\varepsilon \end{bmatrix}\begin{bmatrix} 0 & 0 \\ 1 & 0 \end{bmatrix}\begin{bmatrix} 0 & 1 \\ -1 & -\varepsilon \end{bmatrix}^{-1} \\
&= \begin{bmatrix} -\varepsilon & -1 \\ \varepsilon^2 & \varepsilon \end{bmatrix}
\end{align*}
and observe that, as an element of $\PG[\slt]$, this is $\iota([1:\varepsilon])$.

Thus we now have a complete description of the indecomposable $\slt$-modules.  We finish this section with one more computation that will be needed in \Cref{secLieEx}: the computation of the Heller shifts $\Omega(V(\lambda))$ for indecomposable $V(\lambda)$.  Note that $V(p-1) = Q(p-1)$ is projective so $\Omega(V(p-1)) = 0$.  For other $V(\lambda)$ we have the following.

\begin{Prop} \label{propOmega}
Let $\lambda = rp + a$ be a non-negative integer and $0 \leq a < p$ its remainder modulo $p$.  If $a \neq p - 1$ then $\Omega(V(\lambda)) = V((r + 2)p - a - 2)$.
\end{Prop}
\begin{proof}
This will be a direct computation.  We will determine the projective cover $f\colon P \to V(\lambda)$ and then set $f(x) = 0$ for an arbitrary element $x \in P$.  This will give us the relations determining $\ker f = \Omega(V(\lambda))$ which we will convert into a basis and identify with $V((r + 2)p - a - 2)$.

The indecomposable summands of $P$ are in bijective correspondence with the indecomposable summands (all simple) of the top of $V(\lambda)$, i.e.\ $V(\lambda)/\rad V(\lambda)$.  This correspondence is given as follows, if $\pi_q\colon V(\lambda) \to V(a)$ is the projection onto a summand of the top of $V(\lambda)$ then the projective cover $\phi_a\colon Q(a) \to V(a)$ factors through $\pi_q$.
\[\xymatrix{Q(a) \ar[rr]^{\phi_a} \ar@{-->}[dr]_{f_q} && V(a) \\ & V(\lambda) \ar[ur]_{\pi_q}}\]
The map $f_q\colon Q(a) \to V(\lambda)$ so defined is the restriction of $f\colon P \to V(\lambda)$ to the summand $Q(a)$ of $P$.

The module $V(\lambda)$ fits into a short exact sequence
\[0 \to V(p - a - 2)^{\oplus r + 1} \to V(\lambda) \overset{\pi}{\to} V(a)^{\oplus r + 1} \to 0\]
where $\pi$ has components $\pi_q$ for $q = 0, 1, \ldots, r$.  Each $\pi_q\colon V(\lambda) \to V(a)$ is given by
\[v_i \mapsto \begin{cases} v_{i - qp} & \text{if} \ qp \leq i \leq qp + a \\ 0 & \text{otherwise.} \end{cases}\]
Hence the top of $V(\lambda)$ is $V(a)^{\oplus r + 1}$ and $P = Q(a)^{\oplus r + 1}$.  Recall that $Q(a)$ has basis $\set{v_0, v_1, \ldots, v_{2p - a - 2}} \cup \set{w_{p - a - 1}, w_{p - a}, w_{p - 1}}$.  The map $f_q$ is uniquely determined up to a nonzero scalar and is given by
\begin{align*}
v_i & \mapsto -(a + 1)^2\binom{p - i - 1}{a + 1}v_{(q - 1)p + a + i + 1} && \text{if} \ 0 \leq i \leq p - a - 2, \\
w_i & \mapsto (-1)^{i + a}\binom{a}{i + a + 1 - p}^{-1}v_{(q - 1)p + a + i + 1} && \text{if} \ p - a - 1 \leq i \leq p - 1, \\
v_i & \mapsto 0 && \text{if} \ p - a - 1 \leq i \leq p - 1, \\
v_i & \mapsto (-1)^{a + 1}(a + 1)^2\binom{i + a + 1 - p}{a + 1}v_{(q - 1)p + a + i + 1} && \text{if} \ p \leq i \leq 2p - a - 2.
\end{align*}
This gives $f = [f_0 \ f_1 \ \cdots \ f_r]$.  To distinguish elements from different summands of $Q(a)^{\oplus r + 1}$ let $\set{v_{q,0}, v_{q,1}, \ldots, v_{q,2p - a - 2}} \cup \set{w_{q,p - a - 1}, w_{q,p - a}, w_{q,p - 1}}$ be the basis of the $q^\text{th}$ summand of $Q(a)^{\oplus r + 1}$.  Then any element of the cover can be written in the form
\[x = \sum_{q = 0}^r\left[\sum_{i = 0}^{2p - a - 2}c_{q,i}v_{q,i} + \sum_{i = p - a - 1}^{p - 1}d_{q,i}w_{q,i}\right].\]
for some $c_{q,i}, d_{q,i} \in k$.  Applying $f$ gives
\begin{align*}
f(x) = &\sum_{q = 0}^r\Bigg[-(a + 1)^2\sum_{i = 0}^{p - a - 2}\binom{p - i - 1}{a + 1}c_{q,i}v_{(q - 1)p + a + i + 1} \\
&\hspace{24pt} + (-1)^{a + 1}(a + 1)^2\sum_{i = p}^{2p - a - 2}\binom{i + a + 1 - p}{a + 1}c_{q,i}v_{(q - 1)p + a + i + 1} \\
&\hspace{24pt} + \sum_{i = p - a - 1}^{p - 1}(-1)^{a + i}\binom{a}{i + a + 1 - p}^{-1}d_{q,i}v_{(q - 1)p + a + i + 1}\Bigg]
\end{align*}
Observe that $0 \leq i \leq p - a - 2$ and $p \leq i \leq 2p - a - 2$ give $a + 1 \leq a + i + 1 \leq p - 1$ and $p + a + 1 \leq a + i + 1 \leq 2p - 1$ respectively, whereas $p - a - 1 \leq i \leq p - 1$ gives $p \leq a + i + 1 \leq p + a$.  Looking modulo $p$ we see that the basis elements $v_{(q - 1)p + a + i + 1}$, for $0 \leq q \leq r$ and $p - a - 1 \leq i \leq p - 1$, are linearly independent.  Thus $f(x) = 0$ immediately yields $d_{q,i} = 0$ for all $q$ and $i$.

Now rearranging we have
\begin{align*}
f(x) =& \sum_{q = 0}^r\sum_{i = 0}^{p - a - 2}\Bigg[(-1)^{a + 1}(a + 1)^2\binom{i + a + 1}{i}c_{q,i + p}v_{qp + a + 1 + i} \\
& -(a + 1)^2\binom{p - i - 1}{a + 1}c_{q,i}v_{(q - 1)p + a + 1 + i}\Bigg] \\
=& -(a + 1)^2\sum_{i = 0}^{p - a - 2}\Bigg[\sum_{q = 0}^{r - 1}(-1)^a\binom{i + a + 1}{i}c_{q,i + p}v_{qp + a + 1 + i} \\
& + \sum_{q = 1}^r\binom{p - i - 1}{a + 1}c_{q,i}v_{(q - 1)p + a + 1 + i}\Bigg] \\
=& \sum_{q = 0}^{r - 1}\sum_{i = 0}^{p - a - 2}\Bigg[(-1)^a\binom{i + a + 1}{i}c_{q,i + p} + \binom{p - i - 1}{a + 1}c_{q + 1,i}\Bigg]v_{qp + a + 1 + i}
\end{align*}
so the kernel is defined by choosing $c_{q,i}$, for $0 \leq q \leq r - 1$ and $0 \leq i \leq p - a - 2$, such that
\[(-1)^a\binom{i + a + 1}{i}c_{q,i + p} + \binom{p - i - 1}{a + 1}c_{q + 1,i} = 0.\]
Note that
\[\frac{\binom{p - i - 1}{a + 1}}{\binom{i + a + 1}{i}} = \frac{\binom{p - 1}{i + a + 1}}{\binom{p - 1}{i}} = \frac{(-1)^{i + a + 1}}{(-1)^i} = (-1)^{a + 1}\]
so the above equation simplifies to
\[c_{q,i + p} = c_{q + 1,i}.\]
Thus for $0 \leq i \leq (r + 2)p - a - 2$ the vectors
\[v_i^\prime = \begin{cases} v_{0,i} & \text{if} \ 0 \leq i < p, \\
v_{q,b} + v_{q - 1,p + b} & \text{if} \ 1 \leq q \leq r, \ 0 \leq b \leq p - a - 2, \\
v_{q,b} & \text{if} \ 1 \leq q \leq r, \ p - a - 1 \leq b < p, \\
v_{r,b} & \text{if} \ q = r + 1, \ 0 \leq b \leq p - a - 2. \end{cases}\]
form a basis for the kernel, where $i = qp + b$ with $0 \leq b < p$ the remainder of $i$ modulo $p$.  It is now trivial to check that the $\slt$-action on this basis is identical to that of $V((r + 2)p - a - 2)$.
\end{proof}

\section{Matrix Theorems} \label{secMatThms}

In this section we determine the kernel of four particular maps between free $k[s, t]$-modules.  While these maps are used to represent sheaf homomorphisms in \Cref{secLieEx} we do not approach this section geometrically.  Instead we carry out these computations in the category of $k[s, t]$-modules.

The first map is given by the matrix $M_\varepsilon(\lambda) \in \mathbb M_{rp}(k[s, t])$ shown in \Cref{figMats}.
\begin{sidewaysfigure}[p]
\vspace{350pt}
\[\hspace{0pt}\begin{bmatrix} (a + 2)st & t^2 &&& -(a + 1)\binom{r}{1}\varepsilon^ps^2 &&& -(a + 1)\binom{r}{2}\varepsilon^{2p}s^2 &&& \cdots && -(a + 1)\binom{r}{r}\varepsilon^{rp}s^2 \\
(a + 2)s^2 & (a + 4)st & 2t^2 \\
&(a + 3)s^2 & \ddots & \ddots \\
&& \ddots & \ddots & -t^2 \\
&&& as^2 & ast & 0 \\
&&&& (a + 1)s^2 & \ddots & \ddots \\
&&&&& \ddots & \ddots & -t^2 \\
&&&&&& as^2 & ast & 0 \\
&&&&&&& (a + 1)s^2 & \ddots & \ddots \\
&&&&&&&& \ddots & \ddots & -3t^2 \\
&&&&&&&&& (a - 2)s^2 & (a - 4)st & -2t^2 \\
&&&&&&&&&& (a - 1)s^2 & (a - 2)st & -t^2 \\
&&&&&&&&&&& as^2 & ast
\end{bmatrix}\]
\caption{$M_\varepsilon(\lambda)$} \label{figMats}
\end{sidewaysfigure}
For convenience we index the rows and columns of this matrix using the integers $a + 1, a + 2, \ldots, \lambda$.  Then we can say more precisely that the $(i, j)^\text{th}$ entry of this matrix is
\[M_\varepsilon(\lambda)_{ij} = \begin{cases}
is^2 & \text{if} \ \ i = j + 1 \\
(2i - a)st & \text{if} \ \ i = j \\
(i - a)t^2 & \text{if} \ \ i = j - 1 \\
-(a + 1)\binom{r}{q}\varepsilon^{qp}s^2 & \text{if} \ \ (i, j) = (0, qp + a) \\
0 & \text{otherwise.}
\end{cases}
\]
\begin{Prop} \label{propMker}
The kernel of $M_\varepsilon(\lambda)$ is a free $k[s, t]$-module (ungraded) of rank $r$ whose basis elements are homogeneous of degree $p - a - 2$.
\end{Prop}
\begin{proof}
The strategy is as follows: First we will determine the kernel of $M_\varepsilon(\lambda)$ when considered as a map of $k[s, \frac{1}{s}, t]$-modules.  We do this by exhibiting a basis $H_1, \ldots, H_r$ via a direct calculation.  Then by clearing the denominators from these basis elements we get a linearly independent set of vectors in the $k[s, t]$-kernel of $M_\varepsilon(\lambda)$.  We conclude by arguing that these vectors in fact span, thus giving an explicit basis for the kernel of $M_\varepsilon(\lambda)$ considered as a map of $k[s, t]$-modules.

To begin observe that $s$ is a unit in $k[s, \frac{1}{s}, t]$,
\begin{sidewaysfigure}[p]
\vspace{350pt}
\[\begin{bmatrix} (a + 2)x & x^2 &&& -(a + 1)\binom{r}{1}\varepsilon^p &&& -(a + 1)\binom{r}{2}\varepsilon^{2p} &&& \cdots && -(a + 1)\binom{r}{r}\varepsilon^{rp} \\
a + 2 & (a + 4)x & 2x^2 \\
& a + 3 & \ddots & \ddots \\
&& \ddots & \ddots & -x^2 \\
&&& a & ax & 0 \\
&&&& a + 1 & \ddots & \ddots \\
&&&&& \ddots & \ddots & -x^2 \\
&&&&&& a & ax & 0 \\
&&&&&&& a + 1 & \ddots & \ddots \\
&&&&&&&& \ddots & \ddots & -3x^2 \\
&&&&&&&&& a - 2 & (a - 4)x & -2x^2 \\
&&&&&&&&&& a - 1 & (a - 2)x & -x^2 \\
&&&&&&&&&&& a & ax
\end{bmatrix}\]
\caption{$\frac{1}{s^2}M_\varepsilon(\lambda)$} \label{figxMat}
\end{sidewaysfigure}
thus over this ring the kernel of $M_\varepsilon(\lambda)$ is equal to the kernel of the matrix $\frac{1}{s^2}M_\varepsilon(\lambda)$ (shown in \Cref{figxMat}) with $(i, j)^\text{th}$ entries
\[\frac{1}{s^2}M_\varepsilon(\lambda)_{ij} = \begin{cases}
i & \text{if} \ \ i = j + 1 \\
(2i - a)x & \text{if} \ \ i = j \\
(i - a)x^2 & \text{if} \ \ i = j - 1 \\
-(a + 1)\binom{r}{q}\varepsilon^{qp} & \text{if} \ \ (i, j) = (0, qp + a) \\
0 & \text{otherwise.}
\end{cases}\]
where $x = \frac{t}{s}$.  Let
\[f = \begin{bmatrix} f_{a + 1} \\ f_{a + 2} \\ \vdots \\ f_\lambda \end{bmatrix}\]
be an arbitrary element of the kernel.  Given $i = qp + b$ where $0 \leq b < p$ and $a + 1 \leq i \leq \lambda$ we claim that
\begin{equation} \label{eqPhiForm}
	f_i = (-1)^{\lambda - i}x^{\lambda - i}f_\lambda + (-1)^b\binom{p + a - b}{p - b - 1}x^{p - b - 1}h_{q + 1}
\end{equation}
for some choice of $h_1, \ldots, h_r \in k[s, \frac{1}{s}, t]$ and $h_{r + 1} = 0$.  Moreover, any such choice defines an element $f \in k[s, \frac{1}{s}, t]^{rp}$ such that
\begin{equation} \label{eqQuasKer}
	\frac{1}{s^2}M_\varepsilon(\lambda)f = \begin{bmatrix} \ast \\ 0 \\ \vdots \\ 0 \end{bmatrix}
\end{equation}
holds.

The proof of this claim is by completely elementary methods, we simply induct up the rows of $\frac{1}{s^2}M_\varepsilon(\lambda)$ observing that the condition imposed by each row in \autoref{eqQuasKer} either determines the next $f_i$ or is automatically satisfied allowing us to introduce a free parameter (the $h_i$).

For the base case plugging $i = \lambda$ into \Cref{eqPhiForm} gives $f_\lambda = f_\lambda$.  The condition imposed by the last row in \Cref{eqQuasKer} is $af_{\lambda - 1} + axf_\lambda = 0$ so if $a \neq 0$ then $f_{\lambda - 1} = -xf_\lambda$ and if $a = 0$ then this condition is automatically satisfied.  The formula, when $a = 0$, gives $f_{rp - 1} = -xf_\lambda + h_r$ so we take this as the definition of $h_r$.

Assume the formula holds for all $f_j$ with $j > i \geq a + 1$ and that these $f_j$ satisfy the conditions imposed by rows $i + 2, i + 3, \ldots, \lambda$ of $\frac{1}{s^2}M_\varepsilon(\lambda)$.  Row $i + 1$ has nonzero entries $i + 1$, $(2i - a + 2)x$, and $(i - a + 1)x^2$ in columns $i$, $i + 1$, and $i + 2$ respectively.  First assume $i + 1 \neq 0$ in $k$ or equivalently $b \neq p - 1$ where $i = qp + b$ and $0 \leq b < p$.  Then the condition
\[(i + 1)f_i + (2i - a + 2)xf_{i + 1} + (i - a + 1)x^2f_{i + 2} = 0\]
imposed by row $i + 1$ can be taken as the definition of $f_i$.  Observe that
\begin{align*}
&\frac{-1}{i + 1}\left((-1)^{\lambda - i - 1}(2i - a + 2) + (-1)^{\lambda - i - 2}(i - a + 1)\right) \\
&\qquad = \frac{(-1)^{\lambda - i}}{i + 1}\left((2i - a + 2) - (i - a + 1)\right) \\
&\qquad = \frac{(-1)^{\lambda - i}}{i + 1}\left(i + 1\right) \\
&\qquad = (-1)^{\lambda - i}
\end{align*}
so $f_i = (-1)^{\lambda - i}x^{\lambda - i}f_\lambda + (\text{terms involving } h_j)$.  For the $h_j$ terms there are two cases.  First assume $b < p - 2$.  Then using $\frac{c}{e}\binom{c - 1}{e - 1} = \binom{c}{e}$ we see that
\begin{align*} \label{eqnBinom}
&\frac{-1}{i + 1}\left((-1)^{b + 1}(2i - a + 2)\binom{p + a - b - 1}{p - b - 2} + (-1)^{b + 2}(i - a + 1)\binom{p + a - b - 2}{p - b - 3}\right) \\
&\qquad = \frac{(-1)^b}{b + 1}\left((2b - a + 2)\binom{p + a - b - 1}{p - b - 2} + (p + a - b - 1)\binom{p + a - b - 2}{p - b - 3}\right) \\
&\qquad = \frac{(-1)^b}{b + 1}\left((2b - a + 2)\binom{p + a - b - 1}{p - b - 2} + (p - b - 2)\binom{p + a - b - 1}{p - b - 2}\right) \\
&\qquad = \frac{(-1)^b}{b + 1}(b - a)\binom{p + a - b - 1}{p - b - 2} \\
&\qquad = \frac{(-1)^b}{b + 1}(b + 1)\binom{p + a - b}{p - b - 1} \\
&\qquad = (-1)^b\binom{p + a - b}{p - b - 1}.
\end{align*}
Putting these together we get that
\begin{align*}
f_i &= \frac{-1}{i + 1}((2i - a + 2)xf_{i + 1} + (i - a + 1)x^2f_{i + 2}) \\
&= (-1)^{\lambda - i}x^{\lambda - i}f_\lambda + (-1)^b\binom{p + a - b}{p - b - 1}x^{p - b - 1}h_{q + 1}
\end{align*}
as desired.  Next assume $b = p - 2$ so that $f_{i + 2} = f_{(q + 1)p}$.  The coefficient of $h_{q + 2}$ in $f_{(q + 1)p}$ involves the binomial $\binom{p + a}{p - 1}$.  As $0 \leq a < p - 1$ there is a carry when the addition $(p - 1) + (a + 1) = p + a$ is done in base $p$, thus this coefficient is zero and the $h_j$ terms of $f_i$ are
\begin{align*}
\frac{(-1)^{p}}{i + 1}(2i - a + 2)\binom{a + 1}{0}xh_{q + 1} &= \frac{(-1)^{p - 2}}{2}(a + 2)xh_{q + 1} \\
&= (-1)^b\binom{a + 2}{1}xh_{q + 1}
\end{align*}
as desired.  Thus the induction continues when $i + 1 \neq 0$ in $k$.

Now assume $i + 1 = 0$ in $k$; equivalently, $b = p - 1$.  Then the condition imposed by row $i + 1 = (q + 1)p$ is
\[- axf_{(q + 1)p} - ax^2f_{(q + 1)p + 1} = 0.\]
If $a = 0$ then this is automatic.  If $a > 0$ then there is a base $p$ carry in the addition $(p - 2) + (a + 1) = p + a - 1$, hence
\begin{align*}
&xf_{(q + 1)p} + x^2f_{(q + 1)p + 1} \\
&\qquad = (-1)^{(r - q - 1)p + a}x^{(r - q - 1)p + a + 1}f_\lambda + (-1)^{(r - q - 1)p + a - 1}x^{(r - q - 1)p + a + 1}g \\
&\qquad\quad - \binom{p + a - 1}{p - 2}x^ph_{q + 2} \\
&\qquad = 0
\end{align*}
because the $f_\lambda$ terms cancel and the binomial is zero.  So in either case the condition above is automatic.  The formula for $f_i$ when $i = qp + (p - 1)$ is
\[f_i = (-1)^{(r - q - 1)p + a + 1}x^{(r - q - 1)p + a + 1}g + h_{q + 1}\]
so we take this as the definition of $h_{q + 1}$ and the induction is complete.

Now $f$ must have the given form for some choice of $h_1, \ldots, h_r$ and any such choice gives an element $f$ such that $\frac{1}{s^2}M_\varepsilon(\lambda)f$ is zero in all coordinates save the top ($a + 1$).  All that is left is to impose the condition that $\frac{1}{s^2}M_\varepsilon(\lambda)f$ is zero in the $(a + 1)^\text{th}$ coordinate as well.  This condition is
\[(a + 2)xf_{a + 1} + x^2f_{a + 2} - (a + 1)\sum_{q = 1}^r\binom{r}{q}\varepsilon^{qp}f_{qp + a} = 0.\]
In $(a + 2)xf_{a + 1} + x^2f_{a + 2}$ the $h_j$ terms are
\begin{align*}
&(-1)^{a + 1}\left((a + 2)\binom{p - 1}{p - a - 2} - \binom{p - 2}{p - a - 3}\right)x^{p - a - 1}h_1 \\
&\qquad = (-1)^{a + 1}\left((a + 2)\binom{p - 1}{p - a - 2} + (p - a - 2)\binom{p - 1}{p - a - 2}\right)x^{p - a - 1}h_1 \\
&\qquad = 0
\end{align*}
and the coefficient of the $h_j$ term in $f_{qp + a}$ involves the binomial $\binom{p}{p - a - 1}$ which is zero.  Thus the top row imposes a condition only on $f_\lambda$, and this condition is
\begin{align*}
0 &= (-1)^{rp - 1}(a + 2)x^{rp}f_\lambda + (-1)^{rp - 2}x^{rp}f_\lambda \\
&\quad - (a + 1)\sum_{q = 1}^r(-1)^{(r - q)p}\binom{r}{q}\varepsilon^{qp}x^{(r - q)p}f_\lambda \\
&= (-1)^{rp - 1}(a + 1)\left[\sum_{q = 0}^r(-1)^{qp}\binom{r}{q}\varepsilon^{qp}x^{(r - q)p}\right]f_\lambda.
\end{align*}
Note that $x = \frac{t}{s}$ is algebraically independent over $k$ in $k[s, \frac{1}{s}, t]$ and by hypothesis $a + 1 \neq 0$ in $k$.  The localization of an integral domain is again an integral domain therefore if $f$ is in the kernel then we must have $f_\lambda = 0$.

As the $h_1, \ldots, h_r$ can be chosen arbitrarily this completes the determination of the kernel of $M_\varepsilon(\lambda)$, considered as a map of $k[s, \frac{1}{s}, t]$-modules.  It is free of rank $r$ and the basis elements are given by taking the coefficients of these $h_q$ in \Cref{eqQuasKer}.  Let $H_q$ be the basis element that corresponds to $h_q$; it is shown in \Cref{figHq}.
\begin{sidewaysfigure}[p]
\vspace{400pt}
\[\begin{matrix} a + 1 \\ \vdots \\ qp - 1 \\ qp \\ qp + 1 \\ \vdots \\ qp + a \\ qp + a + 1 \\ \vdots \\ (q + 1)p - 2 \\ (q + 1)p - 1 \\ (q + 1)p \\ \vdots \\ \lambda \end{matrix}\begin{bmatrix} 0 \\ \vdots \\ 0 \\ \binom{a + p}{p - 1}x^{p - 1} \\ -\binom{a + p - 1}{p - 2}x^{p - 2} \\ \vdots \\ (-1)^{p - a - 1}\binom{p}{p - a - 1}x^{p - a - 1} \\ (-1)^{p - a - 2}\binom{p - 1}{p - a - 2}x^{p - a - 2} \\ \vdots \\ -\binom{a + 2}{1}x \\ \binom{a + 1}{0} \\ 0 \\ \vdots \\ 0 \end{bmatrix} = \begin{bmatrix} 0 \\ \vdots \\ 0 \\ 0 \\ 0 \\ \vdots \\ 0 \\ (-1)^{p - a - 2}\binom{p - 1}{p - a - 2}x^{p - a - 2} \\ \vdots \\ -(a + 2)x \\ 1 \\ 0 \\ \vdots \\ 0 \end{bmatrix} \overset{s^{p - a - 2}}{\longrightarrow} \begin{bmatrix} 0 \\ \vdots \\ 0 \\ (-1)^{p - a - 2}\binom{p - 1}{p - a - 2}t^{p - a - 2} \\ \vdots \\ -(a + 2)s^{p - a - 3}t \\ s^{p - a - 2} \\ 0 \\ \vdots \\ 0 \end{bmatrix}\begin{matrix} a + 1 \\ \vdots \\ qp + a \\ qp + a + 1 \\ \vdots \\ (q + 1)p - 2 \\ (q + 1)p - 1 \\ (q + 1)p \\ \vdots \\ \lambda \end{matrix}\]
\caption{$H_q \to s^{p - a - 2}H_q$} \label{figHq}
\end{sidewaysfigure}
I claim that $s^{p - a - 2}H_q$, for $1 \leq q \leq r$, is a basis for the kernel of $M_\varepsilon(\lambda)$, considered as a map of $k[s, t]$-modules.

First note that $H_q$ is supported in coordinates $(q + 1)p - 1$ through $qp + a + 1$.  These ranges are disjoint for different $H_q$ therefore the $s^{p - a - 2}H_q$ are clearly linearly independent.  Let $f \in k[s, t]^{rp}$ be an element of the kernel of $M_\varepsilon(\lambda)$.  Then as an element of $k[s, \frac{1}{s}, t]$ we have that $f$ is in the kernel of $\frac{1}{s^2}M_\varepsilon(\lambda)$ and can write
\[f = \sum_{q = 1}^rc_qH_q.\]
where $c_q \in k[s, \frac{1}{s}, t]$.  The $((q + 1)p - 1)^\text{th}$ coordinate of $f$ is $c_q$ hence $c_q \in k[s, t]$.  Also the $(qp + a + 1)^\text{th}$ coordinate of $f$ is
\[(-1)^{p - a - 2}\binom{p - 1}{p - a - 2}c_qx^{p - a - 2}\]
and the binomial coefficient in that expression is nonzero in $k$ so $c_qx^{p - a - 2} \in k[s, t]$.  In particular, $s^{p - a - 2}$ must divide $c_q$ so write $c_q = s^{p - a - 2}c^\prime_q$ for some $c^\prime_q \in k[s, t]$.  We now have
\[f = \sum_{q = 1}^rc^\prime_qs^{p - a - 2}H_q\]
so the $s^{p - a - 2}H_q$ span and are therefore a basis.  Each $H_q$ is homogeneous of degree $0$ so each $s^{p - a - 2}H_q$ is homogeneous of degree $p - a - 2$.
\end{proof}

The second map we wish to consider is given by the matrix $B(\lambda) \in \mathbb M_{\lambda + 1}(k[s, t])$ shown in \Cref{figMatB}.
\begin{sidewaysfigure}[p]
\vspace{350pt}
\[\begin{bmatrix} \lambda st & \lambda s^2 \\
-t^2 & (\lambda - 2)st & (\lambda - 1)s^2 \\
& -2t^2 & (\lambda - 4)st & (\lambda - 2)s^2 \\
&& -3t^2 & \ddots & \ddots \\
&&& \ddots & \ddots & (\lambda + 2)s^2 \\
&&&& t^2 & (\lambda + 2)st & (\lambda + 1)s^2 \\
&&&&& 0 & \lambda st & \lambda s^2 \\
&&&&&& -t^2 & \ddots & \ddots \\
&&&&&&& \ddots & \ddots & 3s^2 \\
&&&&&&&& -(\lambda - 2)t^2 & -(\lambda - 4)st & 2s^2 \\
&&&&&&&&& -(\lambda - 1)t^2 & -(\lambda - 2)st & s^2 \\
&&&&&&&&&& -\lambda t^2 & -\lambda st
\end{bmatrix}\]
\caption{$B(\lambda)$} \label{figMatB}
\end{sidewaysfigure}
Index the rows and columns of this matrix using the integers $0, 1, \ldots, \lambda$.  Then the entries of $B(\lambda)$ are
\[B(\lambda)_{ij} = \begin{cases}
-it^2 & \text{if} \ \ i = j + 1 \\
(\lambda - 2i)st & \text{if} \ \ i = j \\
(\lambda - i)s^2 & \text{if} \ \ i = j - 1 \\
0 & \text{otherwise.}
\end{cases}\]
\begin{Prop} \label{propBker}
The kernel of $B(\lambda)$ is a free $k[s, t]$-module of rank $r + 1$.  There is one basis element that is homogeneous of degree $\lambda$ and the remaining are homogeneous of degree $p - a - 2$.
\end{Prop}
\begin{proof}
The proof is very similar to the proof of \cref{propMker}.  We start by finding the kernel of the matrix $\frac{1}{s^2}B(\lambda)$ shown in \Cref{figMatBx}
\begin{sidewaysfigure}[p]
\vspace{350pt}
\[\begin{bmatrix} \lambda x & \lambda \\
-x^2 & (\lambda - 2)x & \lambda - 1 \\
& -2x^2 & (\lambda - 4)x & \lambda - 2 \\
&& -3x^2 & \ddots & \ddots \\
&&& \ddots & \ddots & \lambda + 2 \\
&&&& x^2 & (\lambda + 2)x & \lambda + 1 \\
&&&&& 0 & \lambda x & \lambda \\
&&&&&& -x^2 & \ddots & \ddots \\
&&&&&&& \ddots & \ddots & 3 \\
&&&&&&&& -(\lambda - 2)x^2 & -(\lambda - 4)x & 2 \\
&&&&&&&&& -(\lambda - 1)x^2 & -(\lambda - 2)x & 1 \\
&&&&&&&&&& -\lambda x^2 & -\lambda x
\end{bmatrix}\]
\caption{$\frac{1}{s^2}B(\lambda)$} \label{figMatBx}
\end{sidewaysfigure}
whose entries are given by
\[\frac{1}{s^2}B(\lambda)_{ij} = \begin{cases}
-ix^2 & \text{if} \ \ i = j + 1 \\
(\lambda - 2i)x & \text{if} \ \ i = j \\
\lambda - i & \text{if} \ \ i = j - 1 \\
0 & \text{otherwise.}
\end{cases}\]
with $x = \frac{t}{s}$.  Let
\[f = \begin{bmatrix} f_0 \\ f_1 \\ \vdots \\ f_\lambda \end{bmatrix}\]
be an arbitrary element of the kernel.  We induct down the rows of the matrix to show that if $i = qp + b$, where $0 \leq b < p$ then
\[f_{\lambda - i} = (-1)^{\lambda - i}x^{\lambda - i}g + (-1)^b\binom{p + a - b}{p - b - 1}x^{p - b - 1}h_q\]
where $h_r = 0$.

For the base case $i = \lambda$ in the formula gives $f_0 = g$ so we take this as the definition of $g$.  The condition imposed by the first row is $axg + af_1 = 0$ so if $a \neq 0$ then $f_1 = -xg$.  The formula gives $f_1 = -xg + (-1)^a - 1\binom{p + 1}{p - a}x^{p - a}h_r = -xg$ so these agree.  If $a = 0$ then the condition is automatically satisfied and the formula gives $f_1 = -xg + h_{r - 1}$ so we take this as the definition of $h_{r - 1}$.

For the inductive step assume the formula holds for $f_0, f_1, \ldots, f_{i - 1}$ and that these $f_j$ satisfy the conditions imposed by rows $0, \ldots, \lambda - i - 2$.  The three nonzero entries in row $\lambda - i - 1$ are $(b - a + 1)x^2$, $(2b - a + 2)x$, and $b + 1$ in columns $\lambda - i - 2$, $\lambda - i - 1$, and $\lambda - i$ respectively, thus the condition imposed is
\[(b - a + 1)x^2f_{\lambda - i - 2} + (2b - a + 2)xf_{\lambda - i - 1} + (b + 1)f_{\lambda - i} = 0.\]
If $b < p - 2$ then we can solve this for $f_{\lambda - i}$ and we find that it agrees with the formula above (for the $h_j$ terms the computation is identical to the one shown in \cref{propMker}).  If $b = p - 2$ we get
\begin{align*}
f_{\lambda - i} &= -(a + 1)x^2f_{\lambda - i - 2} - (a + 2)xf_{\lambda - i - 1} \\
&= (-1)^{\lambda - i - 1}(a + 1)x^{\lambda - i}g + (-1)^{\lambda - i}(a + 2)x^{\lambda - i}g - (a + 2)h_q \\
&= (-1)^{\lambda - i}x^{\lambda - i}g - \binom{a + 2}{1}xh_q
\end{align*}
as desired.  Finally if $b = p - 1$ then $b + 1 = 0$ in $k$ so the condition is
\[-ax^2f_{\lambda - i - 2} - axf_{\lambda - i - 1} = 0\]
and this is automatically satisfied (the formulas are the same as in \cref{propMker} again).  Thus no condition is imposed on $f_{\lambda - i}$ so we take the formula
\[f_{\lambda - i} = (-1)^{\lambda - i}x^{\lambda - i}g + h_q\]
as the definition of $h_q$.  This completes the induction.

Note that the final row to be $\lambda - i - 1$ we must choose $i = -1$ and therefore are in the case where $b + 1 = 0$ and the condition is automatically satisfied.  The rest of the proof goes as in \cref{propMker}, except that there is no final condition forcing $g = 0$.  If we let $G$ and $H_0, \ldots, H_{r - 1}$ be the basis vectors corresponding to $g$ and $h_0, \ldots, h_{r - 1}$ then the $H_q$ are linearly independent as before.  The first ($0^\text{th}$) coordinate of $G$ is $1$ while the first coordinate of each $H_q$ is $0$ therefore $G$ can be added and this gives a basis for the kernel.  The largest power of $x$ in $G$ is $\lambda$ in the last coordinate and the largest power of $x$ in $H_q$ is $p - a - 2$ in the $(\lambda - qp - a + 1)^\text{th}$ coordinate.  These basis vectors lift to basis vectors of the kernel as a $k[s, t]$-module and are in degrees $\lambda$ and $p - a - 2$ as desired.
\end{proof}

Before we move on to the third map, let us first prove the following lemma which will be needed in \cref{thmFiSimp}.
\begin{Lem} \label{lemBlambda}
Assume $0 \leq \lambda < p$.  Then the $(i, j)^\text{th}$ entry of $B(\lambda)^\lambda$ is contained in the one dimensional space $ks^{\lambda + j - i}t^{\lambda - j + i}$.
\end{Lem}
\begin{proof}
Let $b_{ij}$ be the $(i, j)^\text{th}$ entry of $B(\lambda)$.  By definition the $(i, j)^\text{th}$ entry of $B(\lambda)^\lambda$ is given by
\[(B(\lambda)^\lambda)_{ij} = \sum_{n_1, n_2, \ldots, n_{\lambda - 1}}b_{in_1}b_{n_1n_2}\cdots b_{n_{\lambda - 1}j}.\]
From the definition of $B(\lambda)$ we have
\begin{align*}
b_{ij} \in ks^2 & \ \ \ \text{if} \ j - i = 1, \\ b_{ij} \in kst & \ \ \ \text{if} \ j - i = 0, \\ b_{ij} \in kt^2 & \ \ \ \text{if} \ j - i = -1, \\ b_{ij} = 0 & \ \ \ \text{otherwise.}
\end{align*}
so any given term $b_{in_1}b_{n_1n_2}\cdots b_{n_{\lambda - 1}j}$ in the summation can be nonzero only if the $(\lambda + 1)$-tuple $(n_0, n_1, \ldots, n_\lambda)$ is a \emph{walk} from $n_0 = i$ to $n_\lambda = j$, i.e.\ each successive term of the tuple must differ from the last by at most $1$.  For such a walk we now show by induction that $b_{n_0n_1}b_{n_1n_2}\cdots b_{n_{m - 1}n_m} \in ks^{m + n_m - n_0}t^{m - n_m + n_0}$.  For the base case $m = 1$ we have the three cases above for $b_{n_0n_1}$ and one easily checks that the formula gives $kt^2$, $kst$, or $ks^2$ as needed.

Now assume the statement holds for $m - 1$ so that
\[b_{n_0n_1}\cdots b_{n_{m - 2}n_{m - 1}}b_{n_{m - 1}n_m} \in ks^{m - 1 + n_{m - 1} - n_0}t^{m - 1 - n_{m - 1} + n_0}b_{n_{m - 1}n_m}.\]
There are three cases.  First if $n_m = n_{m - 1} + 1$ then $b_{n_{m - 1}n_m} \in ks^2$ and the set becomes
\[ks^{m - 1 + n_{m - 1} - n_0}t^{m - 1 - n_{m - 1} + n_0}\cdot s^2 = ks^{m - n_m + n_0}t^{m + n_m - n_0}\]
as desired.  Next if $n_m = n_{m - 1}$ then $b_{n_{m - 1}n_m} \in kst$ and the set becomes
\[ks^{m - 1 + n_{m - 1} - n_0}t^{m - 1 - n_{m - 1} + n_0}\cdot st = ks^{m + n_m - n_0}t^{m - n_m + n_0}\]
as desired.  Finally if $n_m = n_{m - 1} - 1$ then $b_{n_{m - 1}n_m} \in kt^2$ and the set becomes
\[ks^{m - 1 + n_{m - 1} - n_0}t^{m - 1 - n_{m - 1} + n_0}\cdot t^2 = ks^{m + n_m - n_0}t^{m - n_m + n_0}\]
as desired.  Thus the induction is complete and for $m = \lambda$ this gives
\[b_{n_0n_1}b_{n_1n_2}\cdots b_{n_{\lambda - 1}n_\lambda} \in ks^{\lambda + n_\lambda - n_0}t^{\lambda - n_\lambda + n_0} = ks^{\lambda + j - i}t^{\lambda - j + i}\]
and completes the proof.
\end{proof}

Moving on, the third map we wish to consider is $B'(\lambda) \in \mathbb M_{rp}(k[s, t])$ defined to be the $rp^\text{th}$ trailing principal minor of $B(\lambda)$, i.e., the minor of $B(\lambda)$ consisting of rows and columns $a + 1, a + 2, \ldots, \lambda$.
\begin{Prop} \label{propBpker}
The kernel of $B'(\lambda)$ is a free $k[s, t]$-module (ungraded) of rank $r$ whose basis elements are homogeneous of degree $p - a - 2$.
\end{Prop}
\begin{proof}
The induction from the proof of \cref{propBker} applies giving
\[f_{\lambda - i} = (-1)^{\lambda - i}x^{\lambda - i}g + (-1)^b\binom{p + a - b}{p - b - 1}x^{p - b - 1}h_q\]
for $0 \leq i < rp$.  All that is left is the condition
\[-(a + 2)xf_{a + 1} - f_{a + 2} = 0\]
from the first row of $\frac{1}{s^2}B'(\lambda)$.  Substituting in the formulas we get
\[(-1)^{a + 1}(a + 1)x^{a + 2}g = 0\]
which forces $g = 0$.  Thus as a basis for the kernel we get $H_0, \ldots, H_{r - 1}$.
\end{proof}

Before we move on to the final map, let us first prove the following lemma which was needed in \Cref{secSl2}

\begin{Lem} \label{lemBJType}
Let $s, t \in k$ so that $B'(\lambda) \in \mathbb M_{rp}(k)$.
\[\jtype(B'(\lambda)) = \begin{cases} [1]^{rp} & \text{if} \ s = t = 0, \\ [p]^{r - 1}[p - a - 1][a + 1] & \text{if} \ s = 0, t \neq 0, \\ [p]^r & \text{if} \ s \neq 0. \end{cases}\]
\end{Lem}
\begin{proof}
If $(s, t) = (0, 0)$ then $B'(\lambda)$ is the zero matrix, hence the Jordan type is $[1]^{rp}$.  If $s = 0$ and $t \neq 0$ then $B'(\lambda)$ only has non-zero entries on the sub-diagonal.  Normalizing these entries to $1$ gives the Jordan form of the matrix from which we read the Jordan type.  If we use the row numbering from $B(\lambda)$ (i.e. the first row is $a + 1$, the second $a + 2$, etc.) then the zeros on the sub-diagonal occur at rows $p, 2p, \ldots, rp$.  Thus the first block is size $p - a - 1$, followed by $r - 1$ blocks of size $p$, and the last block is size $a + 1$.  Hence the Jordan type is $[p]^{r - 1}[p - a - 1][a + 1]$.

Now assume $s \neq 0$.  There are exactly $r(p - 1)$ non-zero entries on the super-diagonal and no non-zero entries above the super-diagonal therefore $\rank B'(\lambda) \geq r(p - 1)$.  But this is the maximal rank that a nilpotent matrix can achieve and such a matrix has Jordan type consisting only of blocks of size $p$.  Hence the Jordan type is $[p]^r$.
\end{proof}

The final map we wish to consider is given by the matrix $C(\lambda) \in \mathbb M_{\lambda + 1}(k[s, t])$ shown in \Cref{figMatB}.
\begin{sidewaysfigure}[p]
\vspace{350pt}
\[\begin{bmatrix} \lambda st & s^2 \\
-\lambda t^2 & (\lambda - 2)st & 2s^2 \\
& -(\lambda - 1)t^2 & (\lambda - 4)st & 3s^2 \\
&& -(\lambda - 2)t^2 & \ddots & \ddots \\
&&& \ddots & \ddots & -s^2 \\
&&&& -(\lambda + 2)t^2 & (\lambda + 2)st & 0 \\
&&&&& -(\lambda + 1)t^2 & \lambda st & s^2 \\
&&&&&& -\lambda t^2 & \ddots & \ddots \\
&&&&&&& \ddots & \ddots & (\lambda - 2)s^2 \\
&&&&&&&& -3t^2 & -(\lambda - 4)st & (\lambda - 1)s^2 \\
&&&&&&&&& -2t^2 & -(\lambda - 2)st & \lambda s^2 \\
&&&&&&&&&& -t^2 & -\lambda st
\end{bmatrix}\]
\caption{$C(\lambda)$} \label{figMatC}
\end{sidewaysfigure}
Index the rows and columns of this matrix using the integers $0, 1, \ldots, \lambda$.  Then the entries of $C(\lambda)$ are
\[C(\lambda)_{ij} = \begin{cases}
(i - \lambda - 1)t^2 & \text{if} \ \ i = j + 1 \\
(\lambda - 2i)st & \text{if} \ \ i = j \\
(i + 1)s^2 & \text{if} \ \ i = j - 1 \\
0 & \text{otherwise.}
\end{cases}\]
\begin{Prop} \label{propCker}
The kernel of $C(\lambda)$ is a free $k[s, t]$-module (ungraded) of rank $r + 1$ whose basis elements are homogeneous of degree $a$.
\end{Prop}
\begin{proof}
Let
\[f = \begin{bmatrix} f_0 \\ f_1 \\ \vdots \\ f_\lambda \end{bmatrix}\]
be an arbitrary element of the kernel of $\frac{1}{s^2}C(\lambda)$ shown in \Cref{figMatCx}
\begin{sidewaysfigure}[p]
\vspace{350pt}
\[\begin{bmatrix} \lambda x & 1 \\
-\lambda x^2 & (\lambda - 2)x & 2 \\
& -(\lambda - 1)x^2 & (\lambda - 4)x & 3 \\
&& -(\lambda - 2)x^2 & \ddots & \ddots \\
&&& \ddots & \ddots & -1 \\
&&&& -(\lambda + 2)x^2 & (\lambda + 2)x & 0 \\
&&&&& -(\lambda + 1)x^2 & \lambda x & 1 \\
&&&&&& -\lambda x^2 & \ddots & \ddots \\
&&&&&&& \ddots & \ddots & \lambda - 2 \\
&&&&&&&& -3x^2 & -(\lambda - 4)x & \lambda - 1 \\
&&&&&&&&& -2x^2 & -(\lambda - 2)x & \lambda \\
&&&&&&&&&& -x^2 & -\lambda x
\end{bmatrix}\]
\caption{$\frac{1}{s^2}C(\lambda)$} \label{figMatCx}
\end{sidewaysfigure}
whose entries are given by
\[\frac{1}{s^2}C(\lambda)_{ij} = \begin{cases}
(i - \lambda - 1)x^2 & \text{if} \ \ i = j + 1 \\
(\lambda - 2i)x & \text{if} \ \ i = j \\
i + 1 & \text{if} \ \ i = j - 1 \\
0 & \text{otherwise.}
\end{cases}\]
with $x = \frac{t}{s}$.  We show by induction that if $i = qp + b$ and $0 \leq b < p$ then
\[f_i = (-1)^b\binom{\lambda}{b}x^bh_q.\]
For the base case the formula gives $f_0 = h_0$ so we take this as the definition of $h_0$.  The condition imposed by row $1$ is $-\lambda xf_0 + f_1 = 0$ which gives $f_1 = -xh_0$ as desired.

For the inductive step assume the formula holds for indices less then $i$ and the condition imposed by all rows of index less than $i - 1$ is satisfied.  Row $i - 1$ has nonzero entries $(i - \lambda - 2)x^2$, $(\lambda - 2i + 2)x$, and $i$ in columns $i - 2$, $i - 1$, and $i$ respectively so the condition is
\[(i - \lambda - 2)x^2f_{i - 2} + (\lambda - 2i + 2)xf_{i - 1} + if_i = 0\]
First assume $i \neq 0, 1$ in $k$.  Then we have
\begin{align*}
f_i &= \frac{-1}{i}\left((-1)^{b - 2}(i - \lambda - 2)\binom{\lambda}{b - 2}x^bh_q + (-1)^{b - 1}(\lambda - 2i + 2)\binom{\lambda}{b - 1}x^bh_q\right) \\
&= \frac{(-1)^b}{b}\left((\lambda - b + 2)\binom{\lambda}{b - 2} + (\lambda - 2b + 2)\binom{\lambda}{b - 1}\right)x^bh_q \\
&= \frac{(-1)^b}{b}((b - 1) + (\lambda - 2b + 2))\binom{\lambda}{b - 1}x^bh_q \\
&= (-1)^b\frac{\lambda - b + 1}{b}\binom{\lambda}{b - 1}x^bh_q \\
&= (-1)^b\binom{\lambda}{b}x^bh_q
\end{align*}
as desired.  Next assume $i = 0$ in $k$.  Then
\begin{align*}
& (i - \lambda - 2)x^2f_{i - 2} + (\lambda - 2i + 2)xf_{i - 1} \\
&\qquad = (\lambda + 2)\binom{\lambda}{p - 2}x^ph_{q - 1} + (\lambda + 2)\binom{\lambda}{p - 1}x^ph_{q - 1}.
\end{align*}
But $a + 1\neq 0$ so $\binom{\lambda}{p - 1} = 0$ and if $a + 1 \neq 0$ then $\binom{\lambda}{p - 2} = 0$ otherwise $\lambda + 2 = 0$.  In any case the above expression is $0$ so the condition imposed by row $i - 1$ is automatically satisfied.  The formula gives $f_i = h_q$ so we take this as the definition of $h_q$.  Finally assume $i = 1$ in $k$.  Then we have
\begin{align*}
f_i &= (\lambda + 1)\binom{\lambda}{p - 1}x^{p + 1}h_{q - 1} - \lambda xh_q \\
&= -\binom{\lambda}{1}xh_q
\end{align*}
as desired.  This completes the induction.  We know that the given formulas for $f_i$ satisfy the conditions imposed by all rows save the last, whose condition is
\[-x^2f_{\lambda - 1} - \lambda xf_\lambda = 0.\]
We have
\[\lambda xf_\lambda = (-1)^a\lambda\binom{\lambda}{a}x^{a + 1}h_r = (-1)^a\lambda x^{a + 1}h_r.\]
If $a = 0$ then
\[x^2f_{\lambda - 1} = (-1)^{p - 1}\binom{\lambda}{p - 1}x^{p + 1}h_{r - 1} = 0\]
and $\lambda = 0$ so this conditions is satisfied.  If $a \neq 0$ then
\[x^2f_{\lambda - 1} = (-1)^{a - 1}\binom{\lambda}{a - 1}x^{a + 1}h_r = (-1)^{a - 1}ax^{a - 1}h_r\]
so
\begin{align*}
& x^2f_{\lambda - 1} + \lambda xf_\lambda \\
& \qquad = (-1)^{a - 1}ax^{a - 1}h_r + (-1)^a\lambda x^{a + 1}h_r \\
& \qquad = (-1)^a(\lambda - a)x^{a + 1}h_r \\
& \qquad = 0
\end{align*}
and the condition is again satisfied so we have found a basis.  If $H_q$ is the basis vector associated to $h_q$ then the smallest and largest powers of $x$ in $H_q$ are $0$ in coefficient $qp$ and $a$ in coefficient $qp + a$.  By the usual arguments the $H_q$ lift to a basis for the kernel of $C(\lambda)$ that is homogeneous of degree $a$.
\end{proof}

The final map we want to consider is parametrized by $0 \leq a < p - 1$.  Given such an $a$, let $D(a) \in \mathbb M_{2p}(k[s, t])$ be the block matrix
\[D(a) = \begin{bmatrix} B(2p - a - 2) & D'(a) \\ 0 & B(a)^\dagger \end{bmatrix}\]
where $D'(a)$ and $B(a)^\dagger$ are as follows.  The matrix $D'(a)$ is a $(2p - a - 1) \times (a + 1)$ matrix whose $(i, j)^\text{th}$ entry is
\[D'(a)_{ij} = \begin{cases} \frac{1}{i + 1}s^2 & \text{if} \ i - j = p - a - 2 \\ \frac{1}{a + 1}t^2 & \text{if} \ (i, j) = (p, a) \\ 0 & \text{otherwise.} \end{cases}\]
The matrix $B(a)^\dagger$ is produced from $B(a)$ by taking the transpose and then swapping the variables $s$ and $t$.
\[B(a)^\dagger = \begin{bmatrix} ast & -s^2 \\ at^2 & (a - 2)st & -2s^2 \\ & (a - 1)t^2 & \ddots & \ddots \\ && \ddots & -(a - 2)st & -as^2 \\ &&& t^2 & -ast \end{bmatrix}\]

\begin{Prop} \label{propQker}
The inclusion of $k[s, t]^{2p - a - 1}$ into $k[s, t]^{2p}$ as the top $2p - a - 1$ coordinates of a column vector induces an isomorphism $\ker B(2p - a - 2) \simeq \ker D(a)$.
\end{Prop}
\begin{proof}
As $D(a)$ is block upper-triangular with $B(2p - a - 2)$ the top most block on the diagonal it suffices to show that every element of $\ker D(a)$ is of the form $\left[\begin{smallmatrix} v \\ 0 \end{smallmatrix}\right]$ with respect to this block decomposition.  That is, we must show that if
\[f = \begin{bmatrix} f_0 \\ f_1 \\ \vdots \\ f_{2p - 1} \end{bmatrix}\]
is an element of $\ker D(a)$ then $f_i = 0$ for all $2p - a - 1 \leq i \leq 2p - 1$.  Obviously it suffices to prove this for $\frac{1}{t^2}D(a)$ over $k[s, t, \frac{1}{t}]$ so let $x = \frac{s}{t}$.

We start by proving that $f_{2p - 1} = 0$.  There are two cases, first assume that $a + 2 = 0$ in $k$.  Then row $p$ of $\frac{1}{t^2}D(a)$ has only one nonzero entry, a $\frac{1}{a + 1}$ in column $2p - 1$.  Thus $f \in \ker \frac{1}{t^2}D(a)$ gives $\frac{1}{a + 1}f_{2p - 1} = 0$ in $k[s, t, ]$, hence $f_{2p - 1} = 0$.  Next assume that $a + 2 < p$.  Then the induction from \cref{propBker} applies to rows $p + 1, \ldots, 20 - a - 2$ and gives
\[f_i = (-1)^{a + i}x^{2p - a - 2 - i}f_{2p - a - 2}\]
for $p \leq i \leq 2p - a - 2$.  The condition imposed by row $p$ is
\[-(a + 2)xf_p - (a + 2)x^2f_{p + 1} + \frac{1}{a + 1}f_{2p - 1} = 0.\]
But note that the induction gave us $f_p = -xf_{p + 1}$ so this simplifies to $\frac{1}{a + 1}f_{2p - 1} = 0$ and again we have $f_{2p - 1} = 0$.

Now the condition imposed by the last row of $D(a)$ gives $f_{2p - 2} = axf_{2p - 1} = 0$.  By induction the $i^\text{th}$ row gives $-if_{i - 1} = (2i + a + 2)xf_i + (i + a + 2)x^2f_{i + 1} = 0$, hence $f_{i - 1} = 0$, for $p - a \leq i \leq 2p - 2$ and this completes the proof.
\end{proof}

\section{Explicit computation of $\gker{M}$ and $\F_i(V(\lambda))$} \label{secLieEx}

In this final section we carry out the explicit computations of the sheaves $\gker{M}$, for every indecomposable $\slt$-module $M$, and $\F_i(V(\lambda))$ for $i \neq p$.  Friedlander and Pevtsova \cite[Proposition 5.9]{friedpevConstructions} have calculated the sheaves $\gker{V(\lambda)}$ for Weyl modules $V(\lambda)$ such that $0 \leq \lambda \leq 2p - 2$.  Using the explicit descriptions of these modules found in \Cref{secSl2} we can do the calculation for the remaining indecomposable modules in the category.

\begin{Prop} \label{thmKer}
Let $\lambda = rp + a$ with $0 \leq a < p$ the remainder of $\lambda$ modulo $p$.  The kernel bundles associated to the indecomposable $\slt$-modules from \cref{thmPremet} are
\begin{align*}
\gker{\Phi_\xi(\lambda)} &\simeq \mathcal O_{\mathbb P^1}(a + 2 - p)^{\oplus r} \\
\gker{V(\lambda)} &\simeq \mathcal O_{\mathbb P^1}(-\lambda) \oplus \mathcal O_{\mathbb P^1}(a + 2 - p)^{\oplus r} \\
\gker{V(\lambda)^\ast} &\simeq \mathcal O_{\mathbb P^1}(-a)^{\oplus r + 1} \\
\gker{Q(a)} &\simeq \mathcal O_{\mathbb P^1}(-a) \oplus \mathcal O_{\mathbb P^1}(a + 2 - 2p)
\end{align*}
\end{Prop}
\begin{proof}
Assume first that $\xi = [1 : \varepsilon]$.  Then using the basis from \Cref{secSl2} we get that the matrix defining $\Theta_{\Phi_\xi(\lambda)}$ has entries
\[(\Theta_{\Phi_\xi(\lambda)})_{ij} = \begin{cases}
ix & \text{if} \ \ i = j + 1 \\
(2i - a)z & \text{if} \ \ i = j \\
(a - i)y & \text{if} \ \ i = j - 1 \\
-(a + 1)\binom{r}{q}\varepsilon^{qp}x & \text{if} \ \ (i, j) = (0, qp + a) \\
0 & \text{otherwise.}
\end{cases}\]
Pulling back along the map $\iota\colon\mathbb P^1 \to \PG[\slt]$ from \cref{exNslt} corresponds with extending scalars through the homomorphism
\begin{align*}
\frac{k[x, y, z]}{(xy + z^2)^n} &\to k[s, t] \\
(x, y, z) &\mapsto (s^2, -t^2, st).
\end{align*}
Thus the matrix of $\Theta_{\Phi_\xi(\lambda)}$ becomes the matrix $M_\varepsilon(\lambda)$ from \cref{propMker} and we see that the kernel is free.  A basis element, homogeneous of degree $m$, spans a summand of the kernel isomorphic to $k[s, t][-m]$.  By definition the $\mathcal O_{\mathbb P^1}$-module corresponding to $k[s, t][-m]$ is $\mathcal O_{\mathbb P^1}(-m)$ so the description of the kernel translates directly to the description of the sheaf above.

The remaining cases are all identical.  The modules $V(\lambda)$, $\Phi_{[0 : 1]}(\lambda)$, $V(\lambda)^\ast$, and $Q(a)$ give the matrices $B(\lambda)$, $B'(\lambda)$, $C(\lambda)$, and $D(a)$ whose kernels are calculated in Propositions \autoref{propBker}, \autoref{propBpker}, \autoref{propCker}, and \autoref{propQker} respectively.
\end{proof}

Next we compute $\F_i(V(\lambda))$ for any $i \neq p$ and any indecomposable $V(\lambda)$.  The proof is by induction on $r$ in the expression $\lambda = rp + a$.  For the base case we start with $V(\lambda)$ a simple module, i.e., $r = 0$.  Note that for the base case we do indeed determine $\F_p(V(\lambda))$, it is during the inductive step that we lose $i = p$.

\begin{Prop} \label{thmFiSimp}
If $0 \leq \lambda < p$ then
\[\F_i(V(\lambda)) = \begin{cases}\gker{V(\lambda)} & \text{if} \ i = \lambda + 1 \\ 0 & \text{otherwise.}\end{cases}\]
\end{Prop}
\begin{proof}
First note that $V(\lambda)$ has constant Jordan type $[\lambda + 1]$ so \cref{thmFi} tells us that when $i \neq \lambda + 1$ the sheaf $\F_i(V(\lambda))$ is locally free of rank $0$, hence is the zero sheaf.

For $i = \lambda + 1$ recall from the previous proof that the map $\Theta_{V(\lambda)}$ of sheaves is given in the category of $k[s, t]$-modules by the matrix $B(\lambda)$ in \Cref{figMatB}.  The $(\lambda + 1)^\text{th}$ power of a matrix of Jordan type $[\lambda + 1]$ is zero so the entries of $B(\lambda)^{\lambda + 1}$ are polynomials representing the zero function.  We assume that $k$ is algebraically closed so this means $B(\lambda)^{\lambda + 1} = 0$ and therefore $\Theta_{V(\lambda)}^{\lambda + 1} = 0$.  In particular
\[\gim[\lambda]{V(\lambda)} \subseteq \gker{V(\lambda)}\]
so the definition of $\F_{\lambda + 1}(V(\lambda))$ gives
\begin{equation} \label{eqnFi}
\F_{\lambda + 1}(V(\lambda)) = \frac{\gker{V(\lambda)} \cap \gim[\lambda]{V(\lambda)}}{\gker{V(\lambda)} \cap \gim[\lambda + 1]{V(\lambda)}} = \gim[\lambda]{V(\lambda)}.
\end{equation}
We have a short exact sequence of $k[s, t]$-modules
\[0 \to \im B(\lambda)^\lambda \to \ker B(\lambda) \to \frac{\ker B(\lambda)}{\im B(\lambda)^\lambda} \to 0.\]
If we show that the quotient $\ker B(\lambda)/\im B(\lambda)^\lambda$ is finite dimensional then by Serre's theorem and \autoref{eqnFi} this gives a short exact sequence of sheaves
\[0 \to \F_i(V(\lambda)) \to \gker{V(\lambda)} \to 0 \to 0\]
and completes the proof.

To show that $\ker B(\lambda)/\im B(\lambda)^\lambda$ is a finite dimensional module note that from $B(\lambda)^{\lambda + 1} = 0$ we get that the columns of $B(\lambda)^\lambda$ are contained in the kernel of $B(\lambda)$ which, in \cref{propBker} we determined is a free $k[s, t]$-module with basis element
\[G = \begin{bmatrix} s^\lambda \\ -s^{\lambda - 1}t \\ \vdots \\ (-1)^\lambda t^\lambda \end{bmatrix}.\]
We also know by \cref{lemBlambda} that the first entry in the $j^\text{th}$ column of $B(\lambda)^\lambda$ is $c_js^{\lambda + j}t^{\lambda - j}$ for some $c_j \in k$, so the $j^\text{th}$ column must therefore be $c_js^jt^{\lambda - j}G$.  The columns of $B(\lambda)^\lambda$ range from $j = 0$ to $j = \lambda$ so this shows that $G$ times any monomial of degree $\lambda$ is contained in the image of $B(\lambda)^\lambda$.  Thus the quotient $\ker B(\lambda)/\im B(\lambda)^\lambda$ is spanned, as a vector space, by the set of vectors of the form $G$ times a monomial of degree strictly less than $\lambda$.  There are only finitely many such monomials therefore $\ker B(\lambda)/\im B(\lambda)^\lambda$ is finite dimensional and the proof is complete.
\end{proof}

Now for the inductive step we will make use of \cref{thmOm}, but in a slightly different form.  Note that the shift in \cref{thmOm} is given by tensoring with the sheaf $\mathcal O_{\PG[\slt]}(1)$ associated to the shifted module $\frac{k[x, y, z]}{xy + z^2}[1]$.  Likewise we consider $\mathcal O_{\mathbb P^1}(1)$ to be the sheaf associated to $k[s, t][1]$.  Pullback through the isomorphism $\iota\colon\mathbb P^1 \to \PG[\slt]$ of \cref{exNslt} yields $\iota^\ast\mathcal O_{\PG[\slt]}(1) = \mathcal O_{\mathbb P^1}(2)$.  Consequently, \cref{thmOm} has the following corollary.

\begin{Cor} \label{corFiOmega}
Let $M$ be an $\slt$-module and $1 \leq i < p$.  With twist coming from $\mathbb P^1$ we have
\[\F_i(M) \simeq \F_{p - i}(\Omega M)(2p - 2i).\]
\end{Cor}

Observe that $i \neq p$ in the theorem; this is why our calculation of $\F_p(V(\lambda))$ for $\lambda < p$ does not induce a calculation of $\F_p(V(\lambda))$ when $\lambda \geq p$.

\begin{Prop}
If $V(\lambda)$ is indecomposable and $i \neq p$ then
\[\F_i(V(\lambda)) \simeq \begin{cases} \mathcal O_{\mathbb P^1}(-\lambda) & \text{if} \ i = \lambda + 1 \pmod p \\ 0 & \text{otherwise.} \end{cases}\]
\end{Prop}
\begin{proof}
Let $\lambda = rp + a$ where $0 \leq a < p$ is the remainder of $\lambda$ modulo $p$.  We prove the result by induction on $r$.  The base case $r = 0$ follows from Propositions \ref{thmKer} and \ref{thmFiSimp}.  For the inductive step assume $r \geq 1$.  By hypothesis the formula holds for $rp - a - 2$ and by \cref{propOmega} we have $\Omega V(rp - a - 2) = V(\lambda)$.  Applying \cref{corFiOmega} we get
\begin{align*}
\F_i(V(\lambda)) &= \F_{p - i}(V(rp - a - 2))(-2i). \\
\intertext{If $i = a + 1$ then}
\F_{a + 1}(V(\lambda)) &= \F_{p - a - 1}(V(rp - a - 2))(-2a - 2) \\
&= \mathcal O_{\mathbb P^1}(a + 2 - rp)(-2a - 2) \\
&= \mathcal O_{\mathbb P^1}(-a - rp) \\
&= \mathcal O_{\mathbb P^1}(-\lambda)
\end{align*}
whereas if $i \neq a + 1$ then $p - i \neq p - a - 1$ so $\F_{p - i}(V(rp - a - 2)) = 0$.  This completes the proof.
\end{proof}

\bibliographystyle{../Refs/alphanum}
\bibliography{../Refs/refs}
\end{document}